\documentclass[11pt]{article}
\usepackage{amsmath,amsfonts,amsthm,amssymb, authblk,color,fullpage}
\usepackage{url}
\usepackage{graphicx,amssymb,amsmath,amsthm}
\usepackage{enumerate}
\usepackage{dsfont}
\usepackage{booktabs}
\usepackage{mathrsfs}
\usepackage{verbatim}
\usepackage{epsfig}
\usepackage{lscape}
\usepackage{subfigure}
\usepackage[colorlinks, linkcolor=red, anchorcolor=blue, citecolor=green]{hyperref}
\usepackage{epstopdf}
\usepackage{color}
\usepackage{caption}
\usepackage{algorithm}
\usepackage{algpseudocode}

\usepackage{graphicx}
\usepackage{subfigure}
\graphicspath{{Figures/}{logo/}}

\newtheorem{definition}{Definition}[section]
\newtheorem{corollary}{Corollary}[section]

\newtheorem{theorem}{Theorem}[section]
\newtheorem{lemma}{Lemma}[section]
\newtheorem{conjecture}{Conjecture}[section]

\newtheorem{remark}{Remark}[section]

\newtheorem{example}{Example}[section]

\theoremstyle{definition}

\theoremstyle{remark}

\newcommand{\Addresses}{{
\bigskip
\footnotesize
Sho Suda, \textsc{Department of Mathematics, National Defense Academy of Japan, Yokosuka, Japan}\par\nopagebreak
\textit{E-mail address}, Sho Suda: \texttt{ssuda@nda.ac.jp}

\medskip

Zili Xu, \textsc{School of Mathematical Sciences, East China Normal University, Shanghai 200241, China; Shanghai Key Laboratory of PMMP, East China Normal University, Shanghai 200241, China; and Key Laboratory of MEA, Ministry of Education, East China Normal University, Shanghai 200241, China}\par\nopagebreak
\textit{E-mail address}, Zili Xu: \texttt{zlxu@math.ecnu.edu.cn}

\medskip

Wei-Hsuan Yu, \textsc{Mathematics Department, National Central University, Taoyuan, Taiwan}\par\nopagebreak
\textit{E-mail address}, Wei-Hsuan Yu: \texttt{u690604@gmail.com}
}}

\date{}

\begin{document}

\title{Existence and nonexistence of spherical $5$-designs of minimal type}
\date{\today}

\author{Sho Suda, Zili Xu, and Wei-Hsuan Yu}

\maketitle

\abstract{

This paper investigates the existence and properties of  spherical $5$-designs of minimal type. We focus on two cases: tight spherical $5$-designs and antipodal spherical $4$-distance $5$-designs. 
We prove that a tight spherical $5$-design is of minimal type if and only if it possesses a specific $Q$-polynomial coherent configuration structure. 
For tight spherical $5$-designs in $\mathbb{R}^d$ of minimal type, we demonstrate that half of the derived code forms an equiangular tight frames (ETF) with parameters $(d-1, \frac{(d-1)(d+1)}{3})$.
This provides a sufficient condition for constructing such ETFs from maximal ETFs with parameters $(d, \frac{d(d+1)}{2})$. 
Moreover, we establish that tight spherical $5$-designs of minimal type cannot exist if the dimension $d$ satisfies a certain arithmetic condition, which holds for infinitely many values of $d$, including $d=119$ and $527$.
For antipodal spherical $4$-distance $5$-designs, we utilize valency theory to derive necessary conditions for certain special types of antipodal spherical $4$-distance $5$-designs to be of minimal type.
}

\section{Introduction}

\subsection{Spherical designs}
A finite set $X\subset \mathbb{S}^{d-1}:=\{x\in\mathbb{R}^d\ |\ \langle x,x\rangle={1}\}$ is called a spherical
$t$-design if the following equality
\begin{equation*}
\int_{\mathbb{S}^{d-1}}f(  {x})d\mu_d(  {x})=\frac{1}{|X|}\sum_{  {x}\in X}^{}f(  {x})
\end{equation*}
holds for all polynomial $f$ of degree at most $t$ \cite{DGS77}. 
Here, $\mu_d$ is the Lebesgue measure on $\mathbb{S}^{d-1}$ normalized by $\mu_d(\mathbb{S}^{d-1})=1$.
It is well known that the size of a spherical $t$-design $X\subset \mathbb{S}^{d-1}$ satisfies the following lower bound \cite{DGS77}:
\begin{equation*}
|X|\geq \begin{cases}
\binom{d+s-1}{s}+\binom{d+s-2}{s-1} & \text{if $t=2s$}	,\\
2\binom{d+s-1}{s} & \text{if $t=2s+1$}	.
\end{cases}	
\end{equation*}
If this bound is attained, then we say that $X$ is a tight spherical $t$-design. 
The existence of tight spherical designs has been studied by bulk of papers \cite{BB09b,BMV04}.

\subsection{Spherical $5$-designs of minimal type}

In this paper we investigate a specific class of spherical $5$-designs, referred to as {\it spherical $5$-designs of minimal type}. 

\begin{definition}{\bf (Spherical $5$-designs of minimal type)}
Let $D\subset\mathbb{S}^{d-1}$ be a spherical $5$-design.
We say $D$ is a spherical $5$-design of minimal type
if there exists an $\alpha\in \mathbb{R}^d$ such that $\langle \alpha,x\rangle\in\{0,\pm 1\}$ for any $x\in D$.
\end{definition}

We remark that spherical $5$-designs of minimal type are closely related to the notions of strongly perfect lattices of minimal type and $m$-stiff configurations. 

\begin{remark}\label{remark-strongly perfect lattices}
A lattice $L$ is called strongly perfect if $\mathrm{Min}(L)$ forms a spherical $4$-design, where
\begin{equation*}
\mathrm{Min}(L)=\{x\in L \mid \langle x,x\rangle=\min(L)\}
\quad\text{and}\quad 
\min(L)=\min_{x\in L,x\neq 0} \langle x,x\rangle.
\end{equation*}
For a strongly perfect lattice $L$ in $\mathbb{R}^d$,
its Berg\'e-Martinet invariant $(\gamma')^2(L)$ satisfies  
\begin{equation}\label{minimal type lattice}
(\gamma')^2(L):=\min(L)\min(L^*)\geq \frac{d+2}{3},	
\end{equation}
where $L^*=\{x\in\mathbb{R}^{d}\mid \langle x,y\rangle\in\mathbb{Z}, \forall y\in L\}$ is the dual lattice of $L$.
The equality in \eqref{minimal type lattice} holds if and only if there exists an $\alpha\in \mathrm{Min}(L^*)$ such that \cite[Theorem 10.4]{Ven01} 
\begin{equation*}
\langle \alpha,x\rangle\in\{0,\pm 1\}, \ \forall x\in \mathrm{Min}(L).
\end{equation*}
A strongly perfect lattice $L$ is called of minimal type if the lower bound \eqref{minimal type lattice} is attained \cite{Ven01,NV00,Mar13}.
It is easy to check that a lattice $L$ is a strongly perfect lattice of minimal type if and only if $\mathrm{Min}(L)$ forms a spherical $5$-design of minimal type.
	
\end{remark}

\begin{remark}
\textcolor{black}{
A finite set $X$ in $S^{d-1}$ is called an $m$-stiff if $X$ is a  spherical $(2m-1)$-design and there exists a vector $z\in S^{d-1}$ such that $|\{\langle z,x\rangle \mid x\in X\}|\leq m$, see \cite{BKN26,B22,B24,FL95}.   
Notably, a spherical $5$-design of minimal type is a special class of a $3$-stiff configuration. 
Recently, Borodachov established a strong connection between $m$-stiff configurations and energy minimization problems \cite{B22,B24}.
The authors in \cite{BKN26} provided some nonexistence results for $m$-stiff when $m$ is large.
However, for small values of $m$, the existence of $m$-stiff configurations remains largely unsolved.
}

\end{remark}

\subsection{Equiangular tight frames}
The optimal line packing problem aims to find a finite set
$X=\{ x_i \}_{i=1}^{n}\subset\mathbb{S}^{d-1}$ with fixed size $n>d$, such that the coherence $\mu(X):=\max\limits_{i\neq j}|\langle x_i, x_j\rangle|$ is minimized (see
\cite{Conway,Fickus2,Haas,Jasper}). 
The following is the Welch bound on the coherence \cite{Welch}: 
\begin{equation}\label{welch}
\mu(X)\,\,\geq\,\, \theta_{n,d}:=\sqrt{\frac{n-d}{d(n-1)}}.
\end{equation}
Denote the angle set of $X$ by $A(X)$, i.e., $A(X):=\{\langle x,y\rangle \mid x,y\in X,x\neq y\}$.
If a set $X\subset\mathbb{S}^{d-1}$ attains the the Welch bound, then we have $d\leq n\leq \frac{d(d+1)}{2}$, and $X$ forms a tight frame with $A(X)=\{\pm\theta_{n,d}\}$. 
For this reason, $X$ is usually referred to as an equiangular tight frame (ETF) with parameters $(d,n)$, and is denoted as ETF$(d,n)$. 

An ETF with parameters $(d,\frac{d(d+1)}{2})$ is called a maximal ETF. 
It is well known that a set $X\subset\mathbb{S}^{d-1}$ is a maximal ETF if and only if $X\cup -X$ forms a tight spherical $5$-design \cite{BB09b}. 
When $d>3$, maximal ETFs exist only when $d=k^2-2$ for some odd integer $k$. 
To date, in the case $d>3$, maximal ETFs are known to exist only
when $(d, n) = (7, 28)$ and $(23, 276)$, respectively. 
Additionally, there are known ETFs with parameters $(d, n) = (6, 16)$ and $(22, 176)$. 
Motivated by these examples, the following conjecture was proposed in \cite[Conjecture~4]{XXY21}.

\begin{conjecture}\label{conj1}
Let $d>3$. The existence of the following are equivalent.
\begin{enumerate}
	\item [\rm (i)] An ETF with parameters $(d,\frac{d(d+1)}{2})$.
	\item [\rm (ii)] An ETF with parameters $(d-1,\frac{(d-1)(d+1)}{3})$.
\end{enumerate}
	
\end{conjecture}

In this paper, we  show that if $X\subset\mathbb{S}^{d-1}$ $(d>7)$ is a maximal ETF and $X\cup -X$ forms a tight spherical $5$-design of minimal type,  
then there exists an ETF with parameters $(d-1,\frac{(d-1)(d+1)}{3})$ (see Corollary \ref{coro1}).  
This provides a sufficient condition for (i) implying (ii) in Conjecture \ref{conj1}.

\subsection{Our contributions}

In this paper we investigate the existence and properties of spherical $5$-designs of minimal type. 
We study the condition when a spherical $5$-design is of minimal type. 
The analysis concentrates on two specific cases: tight spherical $5$-designs and antipodal spherical $4$-distance $5$-designs.

\subsubsection{Tight spherical $5$-designs of minimal type}

We first consider the case when $D\subset\mathbb{R}^d$ is a tight spherical $5$-design. 
Recall that in this case, we have $d=2,3$ or $d=(2m+1)^2-2$ where $m$ is a positive integer,
and we can write $D=X\cup -X$, where $X$ is a maximal ETF with parameters $(d,\frac{d(d+1)}{2})$ \cite{BB09b}.

We establish that $D$ is of minimal type if and only if $D$ has a structure of $Q$-polynomial coherent configurations with specific  parameters described in Theorem \ref{thm:cc} (see  Section~\ref{sec:qcc} for $Q$-polynomial coherent configurations).
We also show that if $D$ is of minimal type, then a half of the derived code of $D$ forms an ETF with parameters $(d-1,\frac{(d-1)(d+1)}{3})$ (see Corollary \ref{coro1}).  
This provides a sufficient condition under which a maximal ETF with parameters $(d,\frac{d(d+1)}{2})$ gives rise to an ETF with parameters $(d-1,\frac{(d-1)(d+1)}{3})$.

Moreover, Theorem \ref{th1} presents a necessary condition for the existence of tight spherical 5-designs of minimal type.
We show that $D$ cannot be of minimal type if the dimension $d$ satisfies certain arithmetic conditions, which hold for infinitely many values of $d$.
For example, Theorem \ref{th1} rules out the possibility of $D$ being of minimal type for $d=119$ and $527$.
Table \ref{table-tight spherical 5 design} summarizes the known results on the existence of tight spherical $5$-designs of minimal type in $\mathbb{R}^d$, where $d=2,3$ and $(2m+1)^2-2$ with $1\leq m\leq 14$.

\begin{table}[!h]
\caption{Tight spherical $5$-designs  in $\mathbb{R}^d$ with $d=2,3$ and $(2m+1)^2-2$ with $1\leq m\leq 14$.}
\centering
\label{table-tight spherical 5 design}
\begin{tabular}{llll}
 $d$  & Tight spherical $5$-design exists?   & Minimal type?    \\ \hline
    2    & Yes \cite{Mun06}     & Yes  [Example \ref{tight-example}]        \\ \hline
    3    & Yes \cite{Mun06}      & No  [Example \ref{tight-example}]       \\ \hline

 7    & Yes \cite{Mun06,CS88}     & Yes  [Example \ref{tight-example}]        \\ \hline
  23    & Yes \cite{Mun06}      & Yes  [Example \ref{tight-example}]       \\ \hline
 47            & No \cite{Mak02,BMV04,NV13}  \\ \hline
 79             & No \cite{BMV04,NV13}    \\ \hline
  119    & Unknown          & No [Theorem \ref{th1}]  \\ \hline
  167        & No \cite{NV13} \\ \hline
  223         & Unknown          & Unknown        \\ \hline
  287       & Unknown        & Unknown                  \\ \hline
  359        & Unknown        &Unknown            \\ \hline
 439        & No \cite{BMV04,NV13}     &    \\ \hline
  527        & Unknown       & No  [Theorem \ref{th1}]  \\ \hline
 623       & No \cite{NV13}        &             \\ \hline
  727       & Unknown    &  Unknown   \\ \hline
  839       & Unknown    & Unknown    \\ \hline
\end{tabular}
\end{table}

\subsubsection{Antipodal spherical $4$-distance $5$-designs of minimal type}

We next focus on the case when $D$ is an antipodal spherical $4$-distance $5$-design. 
Recall that a set $X$ is called an $s$-distance set if its angle set $A(X)$ has size $s$.
We give an equivalent condition on $D$ being of minimal type (see Theorem \ref{thm:levenstein}). 
In section \ref{section4} we utilize valency theory to derive necessary conditions for certain special types of antipodal spherical $4$-distance $5$-designs to be of minimal type.
Table \ref{Levenstein-equality packing} lists the known results on the existence of  antipodal spherical $4$-distance $5$-designs of minimal type.

Note that a spherical $5$-design of minimal type is a special class of $3$-stiff, and the existence of $m$-stiff for small $m$ is largely unknown.
Our results complements the results in \cite{BKN26}.

\begin{table}[!h]
\caption{List of all known antipodal spherical $4$-distance $5$-designs $D$ in $\mathbb{R}^d$.}
\centering
\label{Levenstein-equality packing}
\begin{tabular}{lllll}
 $d$ & $|D|$   & Description  & Minimal type?  \\ \hline
  4  &  24     & $\mathrm{Min}(\mathbb{D}_4)$  & Yes  [Example \ref{antipodal-example}]      \\ \hline
 6  &   72            &  $\mathrm{Min}(\mathbb{E}_6)$  &  Yes [Example \ref{antipodal-example}]   \\ \hline
 7 &   126       & $\mathrm{Min}(\mathbb{E}_7)$   & Yes [Example \ref{antipodal-example}]   \\ \hline
 8 &  240      &  $\mathrm{Min}(\mathbb{E}_8)$      & No [Theorem \ref{tight7design}]      \\ \hline
 $16$  &   $288$     & Maximal real MUBs    &   No  [Example \ref{example-16-144}]    \\ \hline
 16  &    512           & $\mathrm{Min}(\mathbb{O}_{16})$  &  Yes [Example \ref{antipodal-example}]             \\ \hline
22  &   2816       &  $\mathrm{Min}(\mathbb{O}_{22})$    &    Yes [Example \ref{antipodal-example}]        \\ \hline
 23 &  4600       &  $\mathrm{Min}(\mathbb{O}_{23})$     & No [Theorem \ref{tight7design}]        \\ \hline
  $4^s (s\geq3)$  &   $d(d+2)$    & Maximal real MUBs    &   Unknown    \\ \hline
\end{tabular}
\end{table}

\section{Preliminary}

\subsection{Notations}

Let $\text{Hom}_{k}(\mathbb{R}^d)$ be the subspace of all real homogeneous polynomials of degree $k$ on $d$ variables.
Let $\text{Harm}_{k}(\mathbb{R}^d)$ be the  vector space of all real homogeneous harmonic
polynomials of degree $k$ on $d$ variables, equipped with the standard inner product
\begin{equation*}
\langle f,g\rangle=\int_{\mathbb{S}^{d-1}}f(  {x})g(  {x})\text{d}\mu_d(  {x})
\end{equation*}
for $f,g\in\text{Harm}_{k}(\mathbb{R}^d)$. 
It is known that the dimension of $\text{Harm}_{k}(\mathbb{R}^d)$ is $h_k$, where $h_k$ is defined as (see \cite[Theorem 3.2]{DGS77})
\begin{equation*}
h_0:=1,\quad h_1:=d,\quad h_k:=\binom{d+k-1}{k}-	\binom{d+k-3}{k-2},\ \forall k\geq 2.
\end{equation*}
Let $\{\phi_{k,i}^{(d)}\}_{i=1}^{h_k}$ be an orthonormal basis for $\text{Harm}_{k}(\mathbb{R}^d)$.

Let $G_k^{(d)}(x)$ denote the
Gegenbauer polynomial of degree $k$ with the normalization $G_k^{(d)}(1)=h_k$, which can be
defined recursively as follows (see also \cite[Definition 2.1]{DGS77}):
\begin{equation*}
G_0^{(d)}(x):=1,\ G_1^{(d)}(x):=d\cdot x,
\end{equation*}
\begin{equation*}
\frac{k+1}{d+2k}\cdot G_{k+1}^{(d)}(x)=x\cdot G_k^{(d)}(x)-\frac{d+k-3}{d+2k-4}\cdot G_{k-1}^{(d)}(x),\ k\geq 1.
\end{equation*}
The following formulation  is  well-known  \cite[Theorem 3.3]{DGS77}:
\begin{equation*}
G_k^{(d)}(\langle   {x},  {y}\rangle)=\sum\limits_{i=1}^{h_k}\phi_{k,i}^{(d)}(  {x})\phi_{k,i}^{(d)}(  {y}),\
\text{ for }\   {x},  {y}\in\mathbb{S}^{d-1},\  k\in\mathbb{Z}_+.
\end{equation*}

We also need the following notations.

\begin{definition}\label{alldef}
\begin{enumerate}
\item[{\rm (i)}] 
For a finite non-empty set $X\subset\mathbb{S}^{d-1}$, the $k$-th characteristic matrix ${H}_k$ of size $|X|\times h_k$ is defined as  {\rm (see \cite[Definition 3.4]{DGS77})}
\begin{equation*}
{H}_k:=(\phi_{k,i}^{(d)}({x})), \   {x}\in X,\ i\in\{1,2,\ldots,h_k\}.
\end{equation*}

\item[{\rm (ii)}] For mutually disjoint non-empty finite subsets $X_1, X_2,\ldots, X_n\subset \mathbb{S}^{d-1}$, after suitably rearranging the elements of $X=\bigcup_{i=1}^{n}X_i$, we write the $k$-th characteristic matrix $  {H}_{k}$ of $X$ as
\begin{equation*}
{H}_{k}=\begin{pmatrix}
{H}_{k}^{(1)}\\
{H}_{k}^{(2)}\\
\vdots \\
{H}_{k}^{(n)}
\end{pmatrix}=\sum_{i=1}^n  \widetilde{H}_{k}^{(i)},
\end{equation*}
where ${H}_{k}^{(i)}$ is the $k$-th characteristic matrix of $X_i$ and
\begin{equation*}
\widetilde{H}_{k}^{(i)}:=  {e}_i \otimes  {H}_{k}^{(i)}\in\mathbb{R}^{|X|\times h_k}.
\end{equation*}
Here, ${e}_i\in\mathbb{R}^{n}$ is the vector whose $i$-th entry is $1$ and all other entries are $0$. 
For $i\in\{1,2,\ldots,n\}$ we define $  \Delta_{X_i}\in\mathbb{R}^{|X|\times |X|}$ as the diagonal matrix whose $(  {x},  {x})$-entry equals $1$ if $  {x}\in X_i$ and equals zero otherwise. For any $i,j\in\{1,2,\ldots,n\}$ we define $ {J}_{X_i,X_j}\in\mathbb{R}^{|X|\times |X|}$ as the matrix whose $(  {x},  {y})$-entry equals $1$ if $  {x}\in X_i,   {y}\in X_j$ and equals zero otherwise.  

\item[{\rm (iii)}] For two finite sets $X,Y\subset \mathbb{S}^{d-1}$, we define the angle set between $X$ and $Y$ as 
\begin{equation*}
A(X,Y)=\{\langle   {x},  {y}\rangle \mid   {x}\in X,  {y}\in Y,   {x}\neq   {y}\}.
\end{equation*}
If $X=Y$ then we write $A(X)=A(X,X)$.

\item[{\rm (iv)}] For non-empty finite subsets $X_1, X_2,\ldots, X_n\subset \mathbb{S}^{d-1}$, we define the intersection numbers on $X_j$ for $  {x},  {y} \in  \mathbb{S}^{d-1}$ by
\begin{equation*}
p^j_{\alpha,\beta}(  {x},   {y}):=| \{   {z}\in X_j | \langle   {z},  {x}\rangle =\alpha,  \langle   {z},  {x}\rangle =\beta  \} |.
\end{equation*}

\end{enumerate}

\end{definition}

Throughout this paper, we use $I$ to denote the identity matrix of appropriate size, and we denote $\varepsilon_{i,j}:=1-\delta_{i,j}$ where
\begin{equation*}
\delta_{i,j}:=\begin{cases}
1 & \text{if $i=j$},\\
0 & \text{if $i\neq j$}.	
\end{cases}
\end{equation*}

\subsection{Spherical designs}

By the notion of characteristic matrices, the following lemma provides two equivalent definitions of spherical $t$-designs.

\begin{lemma}{\rm (see \cite[Theorem 5.3]{DGS77})}
Let $X\subset\mathbb{S}^{d-1}$ be a non-empty set, and let $H_k$ be the $k$-th characteristic matrix of $X$ for each integer $k\geq 0$.
Then $X$ forms a spherical $t$-design if and only if any one of  the following holds:
\begin{enumerate}[{\rm (i)}]
\item $  {H}_k^\top\cdot   {H}_0=  {0}_{h_k\times 1},\ k=1,2,\ldots,t$.

\item $  {H}_k^\top \cdot   {H}_l=|X|\cdot \delta_{k,l}\cdot    {I}$ when $0\leq k+l\leq t$.

\end{enumerate}
\end{lemma}

The following lemma shows that the {\it derived codes} of a spherical design is also a spherical design.

\begin{lemma}{\rm \cite[Theorem 8.2]{DGS77}}\label{thm:derived-DGS}
Let $X$ be a spherical $t$-design in $\mathbb{S}^{d-1}$ containing the vector $\alpha=(1,0,\ldots,0)$. Assume that $1\leq s^*\leq t+1$ where $s^*:=|A(X)\backslash\{-1\}|$. For any $\beta\in A(X)\backslash\{-1\}$, define the derived code $\mathcal{D}_{\beta}$ with respect to $\alpha$ and $\beta$ as
\begin{equation*}
\mathcal{D}_{\beta}:=\Big\{\xi\in \mathbb{S}^{d-2} \mid (\beta,\sqrt{1-\beta^2}\cdot \xi)\in X\Big\}.
\end{equation*}
Then $\mathcal{D}_{\beta}$ is a spherical $(t+1-s^*)$-design in $\mathbb{S}^{d-2}$ for any $\beta\in A(X)\backslash\{-1\}$.
\end{lemma}

In the following we generalize Lemma \ref{thm:derived-DGS} to the case when $\pm\alpha$ is not contained in $X$.  

\begin{theorem}\label{thm:derived}
Let $X$ be a finite non-empty set in $\mathbb{S}^{d-1}$. 
Let $\alpha=(1,0,\ldots,0)$ and assume that $\pm\alpha\not\in X$. 
Let $B=\{\langle \alpha ,x\rangle \mid x\in X\}=\{\beta_1,\ldots,\beta_s\}$. 
For $i\in \{1,\ldots,s\}$, define the derived code
\begin{equation*}
\mathcal{D}_{\beta_i}=\Big\{\xi\in \mathbb{S}^{d-2} \mid (\beta_i,\sqrt{1-\beta_i^2}\cdot \xi)\in X\Big\}.
\end{equation*}
Then we have the following results.
\begin{enumerate}
\item[\rm (i)] For any $i,j\in\{1,\ldots,s\}$, $A(\mathcal{D}_{\beta_i},\mathcal{D}_{\beta_j})\subset \Big\{\frac{\gamma-\beta_i \beta_j}{\sqrt{(1-\beta_i^2)(1-\beta_j^2)}} \mid \gamma\in A(X)\Big\}$.
\item[\rm (ii)] Assume that $X$ is a spherical $t$-design in $\mathbb{S}^{d-1}$. 
Then $\mathcal{D}_{\beta_i}$ is a $(t+1-s)$-design in $\mathbb{S}^{d-2}$ for any $i\in \{1,\ldots,s\}$. 
\end{enumerate}
\end{theorem}
\begin{proof}

(i) For $x\in \mathcal{D}_{\beta_i}$ and $y\in \mathcal{D}_{\beta_j}$ with $\langle x,y\rangle=\gamma$, write $$
x=(\beta_i,\sqrt{1-\beta_i^2}x'),\quad y=(\beta_j,\sqrt{1-\beta_i^2}y').$$
The condition $\langle x,y\rangle=\gamma$ implies that 
$$
\beta_i \beta_j+\sqrt{(1-\beta_i^2)(1-\beta_j^2)}\langle x',y'\rangle=\gamma,
$$
and thus $$
\langle x',y'\rangle=\frac{\gamma-\beta_i \beta_j}{\sqrt{(1-\beta_i^2)(1-\beta_j^2)}}. 
$$
Therefore $A(\mathcal{D}_{\beta_i},\mathcal{D}_{\beta_j})\subset \Big\{\frac{\gamma-\beta_i \beta_j}{\sqrt{(1-\beta_i^2)(1-\beta_j^2)}} \mid \gamma\in A(X)\Big\}$ holds.

(ii) For any $r$ with $0\leq r \leq t+1-s$, any $F_r\in\text{Hom}_{r}(\mathbb{R}^{d-1})$, and any $k$ with $r\leq k\leq t$, define $G_{r,k}\in\text{Hom}_k(\mathbb{R}^d)$ as 
\begin{equation*}
G_{r,k}(\zeta)=\varepsilon^{k-r}(1-\varepsilon^2)^{r/2}F_r(\xi),
\end{equation*}
where $\zeta=(\varepsilon,\sqrt{1-\varepsilon^2}\xi)$. 
Then 
\begin{equation*}
\sum_{\zeta\in X}G_{r,k}(\zeta)=\sum_{\varepsilon\in B}\varepsilon^{k-r}(1-\varepsilon^2)^{r/2}\sum_{\xi \in \mathcal{D}_{\varepsilon}}F_r(\xi).
\end{equation*} 
For any $T\in O(d-1)$, $T':=\begin{pmatrix} 1 & 0 \\ 0 & T\end{pmatrix}\in O(d)$ fixes $\alpha$. Applying $T'$ to $X$ we obtain 
\begin{equation}\label{T-to-X}
\sum_{\zeta\in X}G_{r,k}(\zeta)=\sum_{\varepsilon\in B}\varepsilon^{k-r}(1-\varepsilon^2)^{r/2}\sum_{\xi \in T\mathcal{D}_{\varepsilon}}F_r(\xi). 
\end{equation}
Here we use the fact that the left-hand side in \eqref{T-to-X} is invariant under the action of $T'$ because $X$ is a spherical design \cite{DGS77}. By taking $k=r,r+1,\ldots,r+s-1$, equation \eqref{T-to-X} yields $s$ linear equations whose $s$ unknowns are 
\begin{equation*}
\sum_{\xi \in T\mathcal{D}_{\varepsilon}}F_r(\xi),\quad \varepsilon \in B.
\end{equation*} 
Its $s\times s$ coefficient matrix is 
\begin{equation*}
[\varepsilon^{k-r}(1-\varepsilon^2)^{r/2}]_{\substack{\varepsilon\in B\\ r\leq k\leq r+s-1}}=[\beta_i^{k}]_{\substack{1\leq i\leq s,\\ 0\leq k\leq s-1}}\cdot \text{diag}[(1-\beta_1^2)^{r/2},\ldots,(1-\beta_s^2)^{r/2}],
\end{equation*}
where $\text{diag}[(1-\beta_1^2)^{r/2},\ldots,(1-\beta_s^2)^{r/2}]$ is an $s\times s$ diagonal matrix whose $(i,i)$-entry is $(1-\beta_i^2)^{r/2}$.
Since $\varepsilon\neq\pm 1$, the coefficient matrix is non-singular.
Therefore, for each $0\leq r \leq t+1-s$, $\sum_{\xi \in T\mathcal{D}_{\varepsilon}}F_r(\xi)$ does not depend on $T\in O(d-1)$. By \cite[Definition 5.1]{DGS77}, this implies that $\mathcal{D}_{\varepsilon}$ is a $(t+1-s)$-design in $\mathbb{S}^{d-2}$. 
\end{proof}

A collection of finite sets $X_1,\ldots,X_n$ is said to be distance-invariant if for any $i,j$ and any $\alpha\in A(X_i,X_j)$, the size of the set 
\[
\{y\in X_j \mid \langle x,y\rangle=\alpha\}
\]
for $x\in X_i$ depends only on $\alpha$, does not depend on the particular choice of $x$.  
The following theorem provides a sufficient condition for a collection of  sets $X_1,\ldots,X_n$ being distance-invariant, which is a generalization of \cite[Theorem 7.4]{DGS77}.
\begin{theorem}\label{thm:di}
    Let $X_i$ be a spherical $t_i$-design in $S^{d-1}$ for $i\in\{1,\ldots,n\}$.  
    Let $s_{i,j}=|A(X_i,X_j)|$. 
    Assume that one of the following holds for $i,j\in\{1,\ldots,n\}$:
    \begin{enumerate}
        \item[{\rm (i)}] $s_{i,j}-1\leq t_j$,
        \item[{\rm (ii)}] $s_{i,j}-2= t_j$ and $X_i=-X_j$. Write $A'(X_i,X_j)=A(X_i,X_j)\setminus\{-1\}$. 
    \end{enumerate}
    Then the sets $\cup_{i=1}^n X_i$ is distance invariant.
\end{theorem}
\begin{proof}
    For $x\in X_i$ and $\alpha\in A(X_i,X_j)$, define $p_{\alpha}^j(x)=|\{z\in X_j \mid \langle x,z\rangle=\alpha \}|$. 
    We reprove equalities \eqref{sum}, \eqref{sperical5design-2}, \eqref{sperical5design-4} and prove them for exponent $\ell$ odd. 
    Let $H_\ell$ be a characteristic matrix of $X_j$. 
    Calculate $(\sum_{\ell=0}^\lambda f_{\lambda,\ell}\phi_{\ell}(x)H_{\ell}^\top)H_0$ in two ways. 

    On the one hand, 
    \begin{align}
        (\sum_{\ell=0}^\lambda f_{\lambda,\ell}\phi_{\ell}(x)H_{\ell}^\top)H_0&=\sum_{\ell=0}^\lambda f_{\lambda,\ell}\phi_{\ell}(x)H_{\ell}^\top H_0\nonumber\\
        &=f_{\lambda,0}\phi_{0}(x)H_{0}^\top H_0\nonumber\\
        &=|X_j|f_{\lambda,0},\label{eq:di1} 
    \end{align}
    and on the other hand, 
    \begin{align}
        (\sum_{\ell=0}^\lambda f_{\lambda,\ell}\phi_{\ell}(x)H_{\ell}^\top)H_0&=\sum_{z\in X_j}\sum_{\ell=0}^\lambda f_{\lambda,\ell}\phi_{\ell}(x)\phi_{\ell}(z)^\top \nonumber\\
        &=\sum_{z\in X_j}\sum_{\ell=0}^\lambda f_{\lambda,\ell}Q_{\ell}(\langle x,z\rangle)\nonumber\\
        &=\sum_{z\in X_j}\langle x,z\rangle^\lambda\nonumber\\
        &=\sum_{\alpha\in A(X_i,X_j)}\alpha^\lambda p_{\alpha}^j(x)+\delta_{X_i,X_j} \label{eq:di2}\\
        &=\sum_{\alpha\in A'(X_i,X_j)}\alpha^\lambda p_{\alpha}^j(x)+\delta_{X_i,X_j}+\delta_{X_i,-X_j}(-1)^\lambda. \label{eq:di3}
    \end{align}
    We consider the cases (i) and (ii) separately. 
    
    (i) We obtain from \eqref{eq:di1} and \eqref{eq:di2}:   
    \begin{align}\label{eq:di4}
        \sum_{\alpha\in A(X_i,X_j)}\alpha^\lambda p_{\alpha}^j(x)=|X_j|f_{\lambda,0}-\delta_{X_i,X_j}.
    \end{align}
    For $\lambda\leq s_{i,j}-1$, \eqref{eq:di4} yields  a system of $s_{i,j}$ linear equations whose unknowns are
    \[
    \{p_{\alpha}^j(x) \mid \alpha\in A(X_i,X_j)\}. 
    \]
    Its coefficients matrix is 
 the Vandermonde matrix $(\alpha^\lambda)$ where $\alpha\in A(X_i,X_j)$ and $\lambda\in\{0,\ldots,s_{i,j}-1\}$. 
 Therefore $p_{\alpha}^j(x)$ does not depend on the choice of $x\in X_i$, and is uniquely determined by $|X_j|$ and $A(X_i,X_j)$. 
 
    (ii) We obtain from \eqref{eq:di1} and \eqref{eq:di3}:   
    \begin{align}\label{eq:di5}
        \sum_{\alpha\in A'(X_i,X_j)}\alpha^\lambda p_{\alpha}^j(x)=|X_j|f_{\lambda,0}-\delta_{X_i,X_j}-(-1)^\lambda.
    \end{align}
    For $\lambda\leq s_{i,j}-2$, \eqref{eq:di5} yields a system of $s_{i,j}-1$ linear equations whose unknowns are
    \[
    \{p_{\alpha}^j(x) \mid \alpha\in A'(X_i,X_j)\}. 
    \]
    Its coefficients matrix is 
 the Vandermonde matrix $(\alpha^\lambda)$ where $\alpha\in A(X_i,X_j)$ and $\lambda\in\{0,\ldots,s_{i,j}-2\}$. 
 Therefore $p_{\alpha}^j$ does not depend on the choice of $x\in X_i$, and is uniquely determined by $|X_j|$ and $A'(X_i,X_j)$.
 Thus, a collection of $X_1,\ldots,X_n$ is distance invariant. 
\end{proof}

\subsection{Association schemes and coherent configurations}\label{sec:qcc}

In this subsection, we introduce the definitions of association schemes, coherent configurations and $Q$-polynomial coherent configuration.

Association schemes are combinatorial axiomatization of transitive finite permutation groups and coherent configurations are that of finite permutation groups.   
A particular class of association schemes, {\it $Q$-polynomial association schemes}, were introduced by Delsarte \cite{D73} to deal with design theory and distance set in a unified way. 
Tight spherical designs can be characterized by the $Q$-polynomial association schemes with certain parameters \cite{DGS77}.
Typical examples of $Q$-polynomial association schemes are obtained from the minimum vectors of tight spherical designs such as the $E_8$ root lattice or the Leech lattice, and tight designs such as Witt designs, and tight orthogonal arrays such as the Golay codes.  
The notion was extended to coherent configurations \cite{Suda22}.

Let $X$ be a non-empty finite set.  We define $\mathrm{diag}(X \times X) = \{(x, x)\ |\ x \in  X\}$. 
For a subset $R$ of $X\times X$, we define $R^\top:=\{(y,x)\ |\ (x,y)\in R\}$, and define the projection of $R$ as follows:
\begin{align*}
\mathrm{pr}_1(R)&=\{x\in X\mid (x,y)\in R \text{ for some }y\in X\},\\
\mathrm{pr}_2(R)&=\{y\in X\mid (x,y)\in R \text{ for some }x\in X\}.
\end{align*}

\begin{definition}{\rm (Coherent configuration)}\label{cc}
Let $X$ be a non-empty finite set and $\mathcal{R}=\{R_i \mid i\in I\}$ be a set of non-empty subsets of $X\times X$.
The pair $\mathcal{C}=(X,\mathcal{R})$ is a coherent configuration if the following properties are satisfied:
\begin{enumerate}
\item[\rm (i)] $\{R_i\}_{i\in I}$ is a partition of $X\times X$,
\item[\rm (ii)] for any $i \in I$, $R_i^\top\in \mathcal{R}$, 
\item[\rm (iii)] $R_i \cap \mathrm{diag}(X\times X)\neq \emptyset$ implies $R_i \subset\mathrm{diag}(X\times X)$,
\item[\rm (iv)] for any $i,j,h\in I$, the number $|\{z\in X \mid (x,z)\in R_i,(z,y)\in R_j\}|$ is independent of the choice of $(x,y)\in R_h$.   
\end{enumerate}
\end{definition}

\begin{remark}
If there is an index $i\in I$ such that $R_i=\mathrm{diag}(X \times X)$, coherent configurations are said to be association schemes and the value $|I|-1$ is said to be the class of the association schemes.  
An association scheme is symmetric if $R_i^\top=R_i$ for any $i\in I$. 
A {\it strongly regular graph} with parameters $(n, k, \lambda, \mu)$ (for short, $\mathrm{srg}(n, k, \lambda, \mu)$) is a graph on $n$ vertices which is regular with valency $k$ and 
has the following two properties:
\begin{itemize}
\item[\rm (i)] any two adjacent vertices have exactly $\lambda$ common neighbours;
\item[\rm (ii)]  any two nonadjacent vertices have exactly $\mu$ common neighbours.
\end{itemize}
For a graph $G=(V,E)$,  $G$ is strongly regular if and only if the pair $(V,\{\mathrm{diag}(X \times X),E, (X\times X)\setminus (\mathrm{diag}(X \times X)\cup E)\})$ is a $2$-classes symmetric association scheme.   
	
\end{remark}

\begin{remark}
Let $A_i$ be the adjacency matrix of the graph $(X,R_i)$. 
We define the coherent algebra $\mathcal{A}$ of the coherent configuration $\mathcal{C}$ as the subalgebra of $\mathrm{Mat}_{|X|}(\mathbb{C})$ generated by $\{A_i \mid i\in I\}$ over $\mathbb{C}$. 
There exists a subset $\Omega$ in $I$ such that $\mathrm{diag}(X\times X)=\bigcup_{i\in \Omega}R_i$ by Definition \ref{cc} (i) and (iii), which is uniquely determined.
We obtain the standard partition $\{X_i\}_{i\in \Omega}$ of $X$ where $X_i=\mathrm{pr}_1(R_i)=\mathrm{pr}_2(R_i)$ for $i\in\Omega$. 
For $i,j\in \Omega$, define  
$I^{(i,j)}=\{R_\ell \mid \ell\in I, R_\ell \subset X_i\times X_j\}$. 
By \cite{Hig75} we know that $\{I^{(i,j)}\mid i,j\in\Omega\}$ is a partition of $I$.
We put $r_{i,j}=|I^{(i,j)}|-\delta_{i,j}$, and we call the matrix $(|I^{(i,j)}|)_{i,j\in\Omega}$ the type of the coherent configuration $\mathcal{C}$. 
By the partition $\{I^{(i,j)}\mid i,j\in\Omega\}$ of $I$, the elements of $I^{(i,j)}$ are renumbered as $R_{\varepsilon_{i,j}}^{(i,j)},\ldots,R_{r_{i,j}}^{(i,j)}$ such that $R_0^{(i,i)}=\mathrm{diag}(X_i\times X_i)$ and $(R_h^{(i,j)})^\top =R_h^{(j,i)}$. 
We denote the adjacency matrix of $R_h^{(i,j)}$ as $A_h^{(i,j)}$. 
For $i,j\in \Omega$, define by $\mathcal{A}^{(i,j)}$ the vector space spanned by  $A_\ell^{(i,j)}$ ($\varepsilon_{i,j}\leq \ell \leq r_{i,j}$) over $\mathbb{C}$. 
Then $\mathcal{A}^{(i,j)}\mathcal{A}^{(j,h)}\subset\mathcal{A}^{(i,h)}$ holds.
We define intersection numbers $p_{\ell,m,n}^{(i,j,h)}$ as 
$$
A_\ell^{(i,j)}A_m^{(j,h)}=\sum\limits_{n=\varepsilon_{i,j}}^{r_{i,j}}p_{\ell,m,n}^{(i,j,h)}A_n^{(i,h)}.
$$
Set $k_\ell^{(i,j)}=p_{\ell,\ell,0}^{(i,j,i)}$.
Then $k_\ell^{(i,j)}=|\{y\in X_j\mid (x,y)\in R_\ell^{(i,j)}\}|$ for any $x\in X_i$.
We call $k_\ell^{(i,j)}$ the valency of $R_\ell^{(i,j)}$.
\end{remark}

We next introduce the definition of the $Q$-polynomial coherent configuration. Let $\widetilde{r}_{i,j}=r_{i,j}-\varepsilon_{i,j}$ for any $i,j\in\Omega$.

\begin{definition}{\rm ($Q$-polynomial coherent configurations)}
Let $\mathcal{C}$ be a coherent configuration such that each fiber $\mathcal{C}^i=(X_i,I^{(i,i)})$ is a symmetric association scheme and there exists a basis $\{  {E}^{(i,j)}_{\ell} \ |\  i, j \in  \Omega, 0\leq\ell  \leq \widetilde{r}_{i,j}\}$ of $\mathcal{A}$ satisfying the following conditions: 
\begin{enumerate}
\item[(B1)] for any $i,j\in\Omega$, $  {E}_0^{(i,j)}=\frac{1}{\sqrt{|X_i||X_j|}}  {J}_{X_i,X_j}$, 
\item[(B2)] for any $i,j\in\Omega$, $\{E_\ell^{(i,j)}\mid 0\leq \ell \leq \widetilde{r}_{i,j}\}$ is a basis of $\mathcal{A}^{(i,j)}$ as a vector space, 
\item[(B3)] for any $i,j\in\Omega,\ell\in \{0,1,\ldots,\widetilde{r}_{i,j}\}$, $(E_\ell^{(i,j)})^\top=E_\ell^{(j,i)}$,
\item[(B4)] for any $i,j,i',j'\in\Omega$ and $\ell\in \{0,1,\ldots,\widetilde{r}_{i,j}\}$, $\ell'\in \{0,1,\ldots,\widetilde{r}_{i',j'}\}$, $E_\ell^{(i,j)}E_{\ell'}^{(i',j')}=\delta_{\ell,\ell'}\delta_{j,i'}E_\ell^{(i,j')}$.
\end{enumerate}
The coherent configuration $\mathcal{C}$ is said to be $Q$-polynomial if for any $i,j\in\Omega$, there exists a set of polynomials $\{v_h^{(i,j)}(x)\mid  0\leq h\leq \widetilde{r}_{i,j}\}$ satisfying that for any $h\in\{0,1,\ldots,\widetilde{r}_{i,j}\}$, $\mathrm{deg}v_h^{(i,j)}(x)=h$ and $\sqrt{|X_i||X_j|}\cdot   {E}_h^{(i,j)}=v_h^{(i,j)}(\sqrt{|X_i||X_j|}\cdot   {E}_1^{(i,j)})$ under the entry-wise product. 

\end{definition}

The author in \cite{Suda10} showed that coherent configurations can be obtained from a union of spherical designs.

\begin{lemma}{\rm \cite[Theorem 2.6]{Suda10}}\label{Suda-thm26}
Let $X_i\subset\mathbb{S}^{d-1}$ be a spherical $t_i$-design for $i\in \{1,\ldots,n\}$. Assume that for any $i, j \in  \{1,\ldots,n\}$ we have $X_i\cap X_j=\emptyset$ or $X_i=X_j$, and $X_i\cap (-X_j) =\emptyset$ or $X_i=-X_j$. 
For $i, j \in  \{1,\ldots,n\}$ we define $\widetilde{s}_{i,j}=|A(X_i, X_j)\backslash \{\pm 1\}|$, $s_{i,j}=|A(X_i, X_j)|$, $\alpha_{i,j}^0=1$, and we write $A(X_i, X_j)=\{\alpha_{i,j}^{1},\ldots,\alpha_{i,j}^{{s}_{i,j}}\}$. When $-1\in A(X_i, X_j)$, we let $\alpha^{s_{i,j}}_{i,j}=-1$. For $i, j \in  \{1,\ldots,n\}$ we define $R^k_{i,j}=\{({x},   {y})\in  X_i\times  X_j | \langle   {x},   {y}\rangle  = \alpha_{i,j}^k\}$ for each integer $1-\delta_{i,j}\leq k\leq s_{i,j}$. If one of the following holds depending on the choice of $i,j,k \in  \{1,\ldots,n\}$: 

\begin{enumerate}
\item[\rm (i)] $\widetilde{s}_{i,j}+\widetilde{s}_{j,k}-2 \leq t_j$;

\item[\rm (ii)] $\widetilde{s}_{i,j}+\widetilde{s}_{j,k}-3 = t_j$, and for any $\gamma \in  A(X_i, X_k)\backslash \{\pm 1\}$ there exist $\alpha \in  A(X_i, X_j)\backslash \{\pm 1\}$, $\beta \in  A(X_j, X_k)\backslash \{\pm 1\}$ such that the intersection number $p^j_{\alpha,\beta}(  {x},   {y})$ is independent of the choice of $  {x}\in X_i,   {y} \in  X_k$ with $\gamma=\langle   {x},  {y}\rangle$;

\item[\rm (iii)] $\widetilde{s}_{i,j}+\widetilde{s}_{j,k}-4 = t_j$, and for any $\gamma \in  A(X_i, X_k)\backslash \{\pm 1\}$ there exist $\alpha,\alpha' \in  A(X_i, X_j)\backslash \{\pm 1\}$, $\beta,\beta' \in  A(X_j, X_k)\backslash \{\pm 1\}$ such that $\alpha\neq \alpha'$, $\beta\neq \beta'$ and the intersection numbers $p^j_{\alpha,\beta}(  {x},   {y})$, $p^j_{\alpha',\beta}(  {x},   {y})$ and $p^j_{\alpha,\beta'}(  {x},   {y})$ are independent of the choice of $  {x}\in X_i,   {y} \in  X_k$ with $\gamma=\langle   {x},  {y}\rangle$;
\end{enumerate}
then $(\bigcup_{i=1}^{n}X_i,\{R^k_{i,j} | 1\leq i, j \leq n, 1-\delta_{i,j}\leq k\leq s_{i,j}\})$ is a coherent configuration.

\end{lemma}

\section{Tight spherical $5$-designs of minimal type}

In this section we consider the existence of tight spherical $5$-designs of minimal type. It is well known that $D\subset \mathbb{S}^{d-1}$ forms a tight spherical $5$-design if and only if $D=X\cup -X$ where $X$ is a maximal ETF in $\mathbb{R}^d$.
Tight spherical $5$-designs may exist only if the dimension $d=2,3$ or $d=(2m+1)^2-2$ where $m$ is a positive integer. The existence of tight spherical $5$-designs is known only for $d=2,3,7$ and $23$ \cite{BB09b}. 

\subsection{Equivalent conditions for tight spherical $5$-designs of minimal type}
In this section we establish several equivalent conditions for tight spherical $5$-designs to be of minimal type.
We first introduce a useful notation.
\begin{definition}
Let $D\subset\mathbb{S}^{d-1}$ be a spherical $5$-design of minimal type. 
Let $\alpha\in \mathbb{R}^d$ be such that $\langle \alpha,x\rangle\in\{0,\pm 1\}$, $\forall x\in D$.
For any $\beta\in\{0,\pm 1\}$, we define the derived code $\mathcal{L}_{\alpha,\beta}(D)\subset \mathbb{S}^{d-1}$ with respect to $\alpha$ and $\beta$ by
\begin{equation*}
\mathcal{L}_{\alpha,\beta}(D):=\Bigg\{\frac{x-\frac{3 \beta}{d+2}\alpha}{\sqrt{1-\frac{3\beta^2}{d+2}}} \ \Big|\    x\in D, \langle \alpha,x\rangle=\beta \Bigg\}\subset \mathbb{S}^{d-1}.
\end{equation*}
\end{definition}

The following is the main theorem of this section, which provides several equivalent conditions on the existence of tight spherical $5$-design of minimal type.

\begin{theorem}\label{thm:cc}
Assume that $d>7$ is an integer. The existence of the following are equivalent.
\begin{enumerate}
\item[{\rm (i)}] A tight spherical $5$-design in $\mathbb{S}^{d-1}$ of minimal type. 
\item[{\rm (ii)}] Spherical $3$-designs $X_1,X_2,X_3$ in $\mathbb{S}^{d-2}$ with $X_1=-X_3,X_1\cap -X_1=\emptyset, X_2=-X_2$, $|X_1|=|X_3|=\frac{(d+1)(d+2)}{6}, |X_2|=\frac{2(d-1)(d+1)}{3}$, and 
\begin{align*}
A(X_1)&=A(X_3)=\left\{\frac{\sqrt{d+2}-3}{d-1},\frac{-\sqrt{d+2}-3}{d-1}\right\},\ A(X_2)=\left\{\frac{1}{\sqrt{d+2}},-\frac{1}{\sqrt{d+2}},-1\right\},\\
A(X_1,X_2)&=A(X_2,X_3)=\left\{\frac{1}{\sqrt{d-1}},-\frac{1}{\sqrt{d-1}}\right\},\ A(X_1,X_3)=\left\{\frac{\sqrt{d+2}+3}{d-1},\frac{-\sqrt{d+2}+3}{d-1},-1\right\}. 
\end{align*}

\item[{\rm (iii)}] A $Q$-polynomial coherent configuration of type $\begin{pmatrix} 
3 & 2 & 3 \\ 
2 & 4 & 2 \\ 
3 & 2 & 3 
\end{pmatrix}$ with the second eigenmatrices:  
\begin{align*}
Q^{(i,i)}&=\left(
\begin{array}{ccc}
1 & d-1 & \frac{1}{6} (d-2) (d-1) \\
1 & \sqrt{d+2}-3 & 2-\sqrt{d+2} \\
1 & -\sqrt{d+2}-3 & \sqrt{d+2}+2 \\
\end{array}
\right) \text{ for }(i,i)\in\{(1,1),(3,3)\}\\
Q^{(i,j)}&=\left(
\begin{array}{ccc}
1 & \sqrt{d-1}  \\
1 & -\sqrt{d-1}  \\
\end{array}
\right) \text{ for }(i,i)\in\{(1,2),(3,2),(2,1),(2,3)\}\\
Q^{(i,j)}&=\left(
\begin{array}{ccc}
1 & -d+1 & \frac{1}{6} (d-2) (d-1) \\
1 & -\sqrt{d+2}+3 & 2-\sqrt{d+2} \\
1 & \sqrt{d+2}+3 & \sqrt{d+2}+2 \\
\end{array}
\right) \text{ for }(i,i)\in\{(1,3),(3,1)\}\\
Q^{(2,2)}&=\left(
\begin{array}{cccc}
1 & d-1 & \frac{1}{3} (d-2) (d+2) & \frac{1}{3} (d-2) (d-1) \\
1 & \frac{d-1}{\sqrt{d+2}} & -1 & -\frac{d-1}{\sqrt{d+2}} \\
1 & -\frac{d-1}{\sqrt{d+2}} & -1 & \frac{d-1}{\sqrt{d+2}} \\
1 & 1-d & \frac{1}{3} (d-2) (d+2) & -\frac{1}{3} (d-2) (d-1) \\
\end{array}
\right).
\end{align*}

\end{enumerate}
\end{theorem}

\begin{proof}
(i)$\Rightarrow$(ii):
Let $D$ be a tight spherical $5$-design in $\mathbb{S}^{d-1}$ of minimal type. Then $D$ has the form $D=X\cup -X$, where $X$ is a maximal ETF with parameters $(d, \frac{d(d+1)}{2})$. 
Since $D$ is a spherical $5$-design, for all $y\in\mathbb{R}^d$ we have \cite{Ven01,BMV04}
\begin{subequations}
\begin{align}
\sum_{x\in D}\langle y,x \rangle^2&=\frac{1}{d}\cdot |D|\cdot \langle y,y \rangle,\label{sperical5design-2} \\
\sum_{x\in D}\langle y,x \rangle^4&=\frac{3}{d(d+2)}\cdot |D|\cdot\langle y,y \rangle^2 \label{sperical5design-4}.
\end{align}
\end{subequations}
Since $D$ is of minimal type, there exists  $ \alpha\in \mathbb{R}^d$ such that $\langle  \alpha, x\rangle\in \{0,\pm1\}$ for all $x\in D$.
For $\ell\in\{0,1\}$, we define 
\begin{equation*}
n_\ell(\alpha):=|\{x\in D \mid \langle \alpha,x\rangle=\pm \ell\}|.
\end{equation*}
Then we have
\begin{equation}\label{sum}
n_0(\alpha)+n_1(\alpha)=|D|=d(d+1),
\end{equation}
and by equation \eqref{sperical5design-2} and \eqref{sperical5design-4} we have
\begin{equation*}
n_1(\alpha)
=\frac{1}{d}\cdot |D|\cdot \langle \alpha,\alpha \rangle
=\frac{3}{d(d+2)}\cdot |D|\cdot\langle \alpha,\alpha \rangle^2.
\end{equation*}
It follows that  
\begin{equation}\label{n0n1}
\langle \alpha,\alpha\rangle=\frac{d+2}{3}, \quad n_0( \alpha)=\frac{2(d-1)(d+1)}{3}\quad\text{and}\quad n_1( \alpha)=\frac{(d+1)(d+2)}{3}. 
\end{equation} 
After suitably transforming the set $D$ and the vector $\alpha$, we may assume that 
\begin{equation*}
\alpha=\bigg(\sqrt{\frac{d+2}{3}},0,\ldots,0\bigg).	
\end{equation*}
Let 
\begin{equation*}
X_1=\mathcal{L}_{\alpha,1}(D),\quad X_2=\mathcal{L}_{\alpha,0}(D)\quad\text{and}\quad X_3=\mathcal{L}_{\alpha,-1}(D).
\end{equation*}
Since $D$ is antipodal, we have $X_3=-X_1$ and $X_2=-X_2$. Moreover, by \eqref{n0n1} we have $|X_1|=|X_3|=n_1( \alpha)/2$ and $|X_2|= n_0( \alpha)$. Theorem \ref{thm:derived} shows that each $X_i$ is a spherical $3$-design in $\mathbb{S}^{d-2}$, and we can check that their angle sets satisfy the desired conditions in (ii). By our assumption that $d>7$, we have $X_1\cap -X_1=\emptyset$. This completes the proof. 

(ii)$\Rightarrow$(i): We proceed the converse implication above. 
Let 
\begin{equation*}
\widetilde{X}_1=\bigg\{\Big(\sqrt{\tfrac{3}{d+2}},\sqrt{\tfrac{d-1}{d+2}}\cdot x\Big)\ \Big|\  x\in X_1\bigg\},\  \widetilde{X}_2=\bigg\{(0,x) \ \Big|\  x\in X_2\bigg\},\ \widetilde{X}_3=-\widetilde{X}_1.
\end{equation*}
Set $D=\bigcup_{i=1}^3 \widetilde{X}_i$.
Then it is routinely shown that $D$ is an antipodal spherical $3$-distance set in $\mathbb{S}^{d-1}$ with attaining the absolute bound for the spherical $3$-distance set. 
Thus it turns out that $D$ is a tight spherical $5$-design in $\mathbb{S}^{d-1}$.

(ii)$\Rightarrow$(iii): We use $\mathcal{C}$ to denote the set $X:=\bigcup_{i=1}^3 X_i$ with binary relations defined from distances. We first prove that $\mathcal{C}$ forms a coherent configuration. 
For $i,j\in\{1,2,3\}$, let 
\begin{equation*}
\widetilde{s}_{i,j}:=|A(X_i,X_j)\backslash\{\pm 1\}|	\quad\text{and}\quad {s}_{i,j}:=|A(X_i,X_j)|.
\end{equation*}
Then we have
\begin{equation*}
(\widetilde{s}_{i,j})_{i,j=1}^3=\begin{pmatrix} 
2 & 2 & 2 \\ 
2 & 2 & 2 \\ 
2 & 2 & 2 
\end{pmatrix}\quad\text{and}\quad ({s}_{i,j})_{i,j=1}^3=\begin{pmatrix} 
2 & 2 & 3 \\ 
2 & 3 & 2 \\ 
3 & 2 & 2 
\end{pmatrix}.
\end{equation*}
Note that for each $i\in\{1,2,3\}$, $X_i$ is a spherical $t_i$-design where $t_i:=3$, so condition (i) in Lemma \ref{Suda-thm26} is satisfied for each $(i,j,h) \in\{1,2,3\}^3$. Combining with the fact that $X_1\cap X_2=\emptyset$, $X_3\cap X_2=\emptyset$ and $X_1=-X_3$, by Lemma \ref{Suda-thm26} we can conclude that the set $X=\bigcup_{i=1}^3 X_i$ with binary relations defined from distances forms a coherent configuration.  
Furthermore, by \cite[Theorem 7.4]{DGS77} and \cite{BB09a} we know that $X_i$ with binary relations defined from distances forms a symmetric association scheme for each $i\in\{1,2,3\}$.

We next construct a basis of $\mathcal{C}$ consisting of primitive idempotents, and show that $\mathcal{C}$ is $Q$-polynomial.
For each $\ell\geq 0$, let $H_{\ell}$ be the $\ell$-th characteristic matrix of $X$, and let $\widetilde{H}_{\ell}^{(i)}$ be defined in Definition \ref{alldef} for each $i\in\{1,2,3\}$.
Following \cite[Theorem 5.10]{Suda22}, we define $  {E}_\ell^{(i,j)}\in\mathbb{R}^{|X|\times |X|}$ for $i,j\in\{1,2,3\}, 0\leq \ell \leq s_{i,j}-\varepsilon_{i,j}$ as follows.
\begin{itemize}
\item For $i,j\in\{1,2,3\}$ and $\ell\in\{0,1\}$,  $E_\ell^{(i,j)}=\frac{1}{\sqrt{|X_i||X_j|}}\widetilde{H}_\ell^{(i)}(\widetilde{H}_\ell^{(j)})^\top$.
\item For $i=j=2$, $E_2^{(2,2)}=\frac{c}{|X_2|}\widetilde{H}_2^{(2)}(\widetilde{H}_2^{(2)})^\top$ for $c=\frac{|X_2|-2}{(d+1)(d-2)}$.  
\item For $i\in\{1,2,3\}$, $E_{s_{i,i}}^{(i,i)}=\Delta_{X_i}-\sum\limits_{k=0}^{s_{i,i}-1}E_k^{(i,i)}$.
\item For $(i,j)\in \{(1,3),(3,1)\}$, $E_{2}^{(i,j)}=A_3^{(i,j)}-\sum\limits_{k=0}^1(-1)^kE_k^{(i,j)}$, where $A_3^{(i,j)}$ is the adjacency matrix defined by inner product $-1$ between $X_i$ and $X_j$. 
\end{itemize}

We next verify that $  {E}_\ell^{(i,j)}\in\mathbb{R}^{|X|\times |X|}$ for $i,j\in\{1,2,3\}, 0\leq \ell \leq s_{i,j}-\varepsilon_{i,j}$ forms a basis that satisfies the condition (B1)-(B4) and the $Q$-polynomial property. It is easy to check that condition (B1) and (B3) hold. We next show condition (B4) holds.

It is clear that $E_\ell^{(i,j)}E_m^{(i',j')}=0$ for $i,j,i',j'\in\{1,2,3\}$ with $j\neq i'$, $\ell\in\{0,1,\ldots,s_{i,j}-\varepsilon_{i,j}\},m\in\{0,1,\ldots,s_{i',j'}-\varepsilon_{i',j'}\}$. 
In the following we will show that for $i,j,h\in\{1,2,3\}$ and $\ell\in\{0,1,\ldots,s_{i,j}-\varepsilon_{i,j}\},m\in\{0,1,\ldots,s_{j,h}-\varepsilon_{j,h}\}$, $  {E}_\ell^{(i,j)}  {E}_m^{(j,h)}=\delta_{\ell,m}  {E}_\ell^{(i,h)}$. 

Define the sets
\begin{equation*}
I_1:=\{1,2,3\}^3\setminus \{(1,3,1),(2,2,2),(3,1,3)\} \quad \text{and}\quad	I_2:=\{(1,3,1),(2,2,2),(3,1,3)\}.
\end{equation*}
Note that $s_{i,j}+s_{j,h}-2\leq t_j$ for each $(i,j,h)\in I_1$, and that $s_{i,j}+s_{j,h}-3= t_j$ for each $(i,j,h)\in I_2$. In the same manner \cite[Theorem 5.10]{Suda22}, one can check that  $  {E}_\ell^{(i,j)}  {E}_m^{(j,h)}=\delta_{\ell,m}  {E}_\ell^{(i,h)}$ holds for
\begin{itemize}
\item $(i,j,h)\in I_1$ and any possible $\ell,m$, 
\item $(i,j,h)\in I_2$ and any $\ell\in \{0,1,\ldots,s_{i,j}-1\},m\in\{0,1,\ldots,s_{j,h}-1\}$ with $(\ell,m)\neq (s_{i,j}-1,s_{j,h}-1)$. 
\end{itemize}
We deal with the following cases. 
\begin{itemize} 
\item The case where $(i,j,h)=(1,3,1)$ and $\ell=s_{1,3}-1,m=s_{3,1}-1$. In this case $A_3^{(i,j)}\widetilde{H}_k^{(j)}=(-1)^k\cdot\widetilde{H}_k^{(i)}$. 
Therefore $E_\ell^{(i,i)}E_m^{(i,i)}=\delta_{\ell,m}E_\ell^{(i,i)}$ for $i\in\{1,3\}$ if and only if $E_\ell^{(i,j)}E_m^{(j,h)}=\delta_{\ell,m}E_\ell^{(i,h)}$ for $i,j,h\in\{1,3\}$.  
\item The case where $(i,j,h)=(2,2,2)$ and $\ell=s_{1,3}-1,m=s_{3,1}-1$ follows from the fact that $X_2$ is a $Q$-polynomial association scheme and the matrices $E_\ell^{(2,2)}$ ($\ell\in\{0,1,2,3\}$) are primitive idempotents \cite{BB09a}.  
\item The case where $(i,j,h)=(3,1,3)$ and $\ell=s_{1,3}-1,m=s_{3,1}-1$ follows similarly as the case where $(i,j,h)=(1,3,1)$ and $\ell=s_{1,3}-1,m=s_{3,1}-1$. 
\end{itemize}
Therefore, we obtain that condition (B4) holds. Using a similar analysis in \cite{Suda22}, one can show that condition (B2) holds. It remains to check the $Q$-polynomial property.

We define the polynomial ${v}_\ell^{(i,j)}(x)$ for $i,j\in\{1,2,3\}, 0\leq \ell \leq s_{i,j}-\varepsilon_{i,j}$ as follows.
\begin{itemize}
\item For $i,j\in\{1,2,3\}$ and $\ell\in\{0,1\}$,  $v_\ell^{(i,j)}(x)=G_\ell^{(d)}(\frac{x}{d})$.
\item For $i=j=2$, $v_2^{(2,2)}(x)=c\cdot G_2^{(d)}(\frac{x}{d})$ for $c=\frac{|X_2|-2}{(d+1)(d-2)}$. 
\item For $i\in\{1,3\}$, $v_{2}^{(i,i)}(x)=|X_i|\cdot F_i(\frac{x}{d})-G_0^{(d)}(\frac{x}{d})-G_1^{(d)}(\frac{x}{d})$. For $i=2$, $v_{3}^{(i,i)}(x)=|X_i|\cdot F_i(\frac{x}{d})-G_0^{(d)}(\frac{x}{d})-G_1^{(d)}(\frac{x}{d})-c\cdot G_2^{(d)}(\frac{x}{d})$. Here, $c=\frac{|X_2|-2}{(d+1)(d-2)}$ and  $F_i(x):=\prod_{\alpha\in A(X_i)}\frac{x-\alpha}{1-\alpha}$ for each $i\in\{1,2,3\}$.
\item For $(i,j)\in \{(1,3),(3,1)\}$, $v_{2}^{(i,j)}(x)=F_3^{(i,j)}(\frac{x}{d})-G_0^{(d)}(\frac{x}{d})-G_1^{(d)}(\frac{x}{d})$, where $F_3^{(i,j)}(x):=\prod_{\alpha\in A(X_i,X_j),\alpha\neq -1}\frac{x-\alpha}{-1-\alpha}$. 
\end{itemize}
Then we can check that the $Q$-polynomial property holds. 

(iii)$\Rightarrow$(ii): 
Consider the matrix $G=\frac{1}{3}\sum_{i,j=1}^3 E_1^{(i,j)}$. Since $G^\top=G$ and $G^2=G$, the matrix $G$ is positive semi-definite. 
The rank of $G$ is $\mathrm{rank}\ G=\mathrm{tr}(G)=\sum_{i=1}^3 \frac{1}{3}\mathrm{tr}E_{1}^{(1,1)}=m_1^{(1,1)}=d-1$. 
Then there is a finite set $X=\bigcup_{i=1}^3 X_i$ in $\mathbb{R}^{d-1}$ with its gram matrix $G$ such that the gram matrix of $X_i$ is $E_1^{(i,i)}$. 
Then the inner product between $X_i$ and $X_j$ ($i\neq j$) appears in the entries of $E_1^{(i,j)}$, and by \cite{BI84} we have $|X_1|=|X_3|=\frac{(d+1)(d+2)}{3},|X_2|=\frac{2(d-1)(d+1)}{3}$. 
Therefore $X_i$ ($i\in\{1,2,3\}$) are the desired subsets in $\mathbb{S}^{d-1}$. 
\end{proof}

In the following we examine whether the known tight spherical $5$-designs are of minimal type.
Recall that tight spherical $5$-designs are known to exist only for $d=2,3,7$ and $23$ \cite{BB09b}. 

\begin{example}\label{tight-example}
When $d=2$, the vertices of a regular $6$-gon forms a tight spherical $5$-design, and one can easily check that it is of minimal type. 
When $d=3$, twelve vertices of an icosahedron on $\mathbb{S}^2$ form a tight spherical $5$-design, and it is not minimal type, because in this case we can calculate that $n_{0}(\alpha)$ and $n_{1}(\alpha)$ in \eqref{n0n1} are not integers.
The tight spherical $5$-designs in $\mathbb{R}^7$ and $\mathbb{R}^{23}$ are the shortest vectors of the lattices $\mathbb{E}_7^*$ and $\mathbb{Q}_{23}(6)^{+2}$, respectively \cite{GY18,CS88,Neumaier,LY}.
These two lattices are both strongly perfect lattices of minimal type since their   
	Berg\'e-Martinet invariants achieve the lower bound \eqref{minimal type lattice} \cite{CS88}.
Hence,  the tight spherical $5$-designs in $\mathbb{R}^7$ and $\mathbb{R}^{23}$ are of minimal type.
	
\end{example}

Theorem \ref{thm:cc} implies the following corollary, which shows that each tight spherical $5$-design of minimal type in $\mathbb{S}^{d-1}$  gives rise to an ETF with parameters $(d-1,\frac{(d-1)(d+1)}{3})$ and a strongly regular graph with parameters in \eqref{srg-para2}. 
In particular, Corollary \ref{coro1} gives a sufficient condition for (i) implying (ii) in Conjecture \ref{conj1}.
   
\begin{corollary}\label{coro1}
Let $d=k^2-2$ where $k>3$ is an odd integer. Assume that there exists a tight spherical $5$-design $D$ in $\mathbb{S}^{d-1}$ of minimal type. Then,

\begin{enumerate}
\item[\rm (i)] there exists an ETF with parameters $(d-1,\frac{(d-1)(d+1)}{3})$;
\item[\rm (ii)] there exists a strongly regular graph with parameters 
\begin{equation}\label{srg-para2}
\Big(\frac{1}{6}k^2(k^2-1),\frac{1}{12}(k-1)(k-2)(k^2-3),\frac{1}{24}(k+2)(k-1)(k-3)(k-5),\frac{1}{24}(k^2-1)(k-2)(k-3)\Big).
\end{equation}

\end{enumerate}

\end{corollary}

\begin{proof}
(i)
Let $\alpha\in \mathbb{R}^d$ be such that $\langle \alpha,x\rangle\in\{0,\pm 1\}$, $\forall x\in D$.
By the proof of Theorem \ref{thm:cc}, a half of the derived code $\mathcal{L}_{\alpha,0}(D)$ has size $\frac{(d-1)(d+1)}{3}$ and the angle set $\{\pm\frac{1}{\sqrt{d+2}}\}$, so it attains the Welch bound \eqref{welch} and forms an ETF with parameters $(d-1,\frac{(d-1)(d+1)}{3})$. 

(ii) The authors in \cite[Proposition 3.2]{BGOY15} showed that, for  a two-distance tight frame $Y=\{y_i\}_{i=1}^n\subset\mathbb{S}^{l-1}$ with $A(Y)=\{a,b\}$, $a^2\neq b^2$, if we construct a graph $G$ with $n$ vertices where vertex $i$ and vertex $j$ are adjacency if $\langle y_i,y_j\rangle=a$, then $G$ forms a strongly regular graph. The parameters $(n,n_a,\lambda,\mu)$ of $G$ can be calculated as
\begin{equation}\label{srg-para}
n_a=\frac{\frac{n}{l}-1-(n-1)b^2}{a^2-b^2},\quad \lambda=\frac{(\frac{n}{l}-2)a-2(n_a-1)ab-(n-2\cdot n_a)\cdot b^2}{(a-b)^2},\quad \mu=\frac{n_a(n_a-\lambda-1)}{n-n_a-1}.
\end{equation}
Note that $X_1$ in Theorem \ref{thm:cc} (ii) is a spherical $3$-design with $|A(X_1)|=2$, so it actually forms a two-distance tight frame. 
Then, substituting $l=d-1$, $n=|X_1|=\frac{(d+1)(d+2)}{6}$, $a=\frac{-\sqrt{d+2}-3}{d-1}=\frac{-k-3}{d-1}$ and $b=\frac{\sqrt{d+2}-3}{d-1}=\frac{k-3}{d-1}$ into \eqref{srg-para}, we see that $X_1$ gives rise to a strongly regular graph with the parameters described in (ii).

\end{proof}

\begin{remark}
The existence of {\rm ETFs} with parameters $(d-1,\frac{(d-1)(d+1)}{3})$ is known only when $d=7$ and $23$, which are {\rm ETF}$(6,16)$ and {\rm ETF}$(22,176)$.
The strongly regular graphs with parameters in \eqref{srg-para2} are known to exist only when $k=3$ and $5$, which are $\mathrm{srg}(12,1,0,0)$ and $\mathrm{srg}(100,22,0,6)$, respectively. The existence of such strongly regular graphs are open for each odd integer $k\geq 7$. For example, the first two open cases are $\mathrm{srg}(392,115,18,40)$ and $\mathrm{srg}(1080,364,88,140)$.
\end{remark}

\subsection{Necessary conditions for tight spherical $5$-designs of minimal type }

In this section we present a necessary condition for a tight spherical $5$-design being of minimal type.
Theorem \ref{th1} rules out the possibility that tight spherical $5$-designs in $\mathbb{R}^{(2m+1)^2-2}$ are of minimal type for $m=5,11,21,29,...$. 

\begin{theorem}\label{th1}
Assume that $m$ is an odd integer satisfying $m\not\equiv1\ (\text{mod 3})$. Assume that $m(m+1)$ is not divisible by the square of an odd prime, and that $m+1$ is not a multiple of 8. Let $d=(2m+1)^2-2$. If there exists a tight spherical $5$-design in $\mathbb{R}^d$, then it is not of minimal type.

\end{theorem}

To prove Theorem \ref{th1}, we first introduce some notations and lemmas from \cite{BMV04,NV13}.
For any $k>0$ and for any positive integer $d$, we denote
\begin{equation*}
\mathbb{S}^{d-1}(k):=\{x\in\mathbb{R}^d\ |\ \langle x,x\rangle={k}\}.	
\end{equation*}
A finite subset $D\subset \mathbb{S}^{d-1}(k)$ is called a spherical $t$-design if $\frac{1}{\sqrt{k}}D$ is a spherical $t$-design in the unit sphere $\mathbb{S}^{d-1}$. 

Let $d=k^2-2$, where $k=2m+1\geq 3$ is an odd integer.
Assume that $D=X\cup-X\subset\mathbb{S}^{d-1}(k)$ is a tight spherical $5$-design, where $X$ is a subset of $\mathbb{S}^{d-1}(k)$ such that $\frac{1}{\sqrt{k}} X$ is a maximal ETF in $\mathbb{R}^d$. Then we have
\begin{equation*}
\langle x,y\rangle=\pm 1,\ \forall\ x,y\in X, x\neq y.
\end{equation*}
Since $D\subset\mathbb{S}^{d-1}(k)$ is a spherical $5$-design, by \eqref{sperical5design-2} and \eqref{sperical5design-4} we have
\begin{subequations}
\begin{align}
\sum_{x\in D}\langle y,x \rangle^2&=\frac{k}{d}\cdot |D|\cdot \langle y,y \rangle,\label{sperical5design-23} \\
\sum_{x\in D}\langle y,x \rangle^4&=\frac{3k^2}{d(d+2)}\cdot |D|\cdot\langle y,y \rangle^2 \label{sperical5design-43}
\end{align}
\end{subequations}
for all $y\in\mathbb{R}^d$.
Let $\Lambda$ be the lattice generated by $X$, and let $\Lambda^*$ denote the dual lattice of $\Lambda$. 
Define the sublattice $\Lambda_+$ of $\Lambda$ by
\begin{equation*}
\Lambda_+:=\bigg \{\sum\limits_{x\in X}^{}c_{x}x\  |\  c_x\in\mathbb{Z},\sum\limits_{x\in X}c_x\equiv0\ (\text{mod 2}) \bigg \}
\end{equation*}
and set 
\begin{equation*}
\Gamma:=\frac{1}{\sqrt{2}}\Lambda_+.
\end{equation*}
A direct calculation shows that 
\begin{equation}
\langle \lambda,x\rangle\equiv0\ \text{(mod 2)}, \ \forall \ \lambda\in \Lambda_+,\ x\in X,\label{l+}
\end{equation}
(see also \cite[equation (21)]{BMV04}).
We need the following two lemmas.

\begin{lemma}{\rm{\cite[Lemma 4.5]{NV13}}}\label{lemma1}
Assume that $m$ is an odd integer. Assume that $m(m+1)$ is not divisible by the square of an odd prime, and that $m+1$ is not a multiple of 8.  Then we have 
\begin{equation*}
\Gamma^*\slash\Gamma\cong\mathbb{Z}\slash2\mathbb{Z}.
\end{equation*}
\end{lemma}

\begin{lemma}{\rm{\cite[Lemma 3.7]{BMV04}}}\label{lemma2}
Assume that $m(m+1)$ is not a multiple of 8.  Then we have 
\begin{equation*}
\langle \lambda,\lambda\rangle\equiv0\ \text{(mod 4)}, \ \forall \ \lambda\in \Lambda_+.
\end{equation*}
In particular, $\Gamma$ is an even lattice.
\end{lemma}

Now we can present a proof of Theorem \ref{th1}.

\begin{proof}[Proof of Theorem \ref{th1}]
We proceed by contradiction. Let $D=X\cup-X\subset\mathbb{S}^{d-1}(k)$ be a tight spherical $5$-design, where $X$ is a subset of $\mathbb{S}^{d-1}(k)$ such that $\frac{1}{\sqrt{k}} X$ is a maximal ETF in $\mathbb{R}^d$. Let $\Lambda$ be the lattice generated by $X$, and let $\Lambda^*$ denote the dual lattice of $\Lambda$. For the aim of contradiction, we assume that there exists $\alpha\in\Lambda^*$ such that $\langle \alpha,x\rangle\in\{0,\pm 1\}$ for all $x\in X$. 
By \eqref{sperical5design-23} and \eqref{sperical5design-43}, we have $\langle\alpha,\alpha\rangle=\frac{d+2}{3k}=\frac{2m+1}{3}$. Noting that $m\not\equiv1\ (\text{mod 3})$, we have
\begin{equation}\label{eq:contradiction}
\langle\alpha,\alpha\rangle=\frac{2m+1}{3}\notin\mathbb{Z}.
\end{equation}
We claim that, under our assumption on $m$, we have $\Lambda^*=\frac{1}{2}\Lambda_+$. Then, according to Lemma \ref{lemma2}, we have
\begin{equation*}
\langle \beta,\beta\rangle\in\mathbb{Z},\quad\forall \beta\in\Lambda^* .
\end{equation*}
This contradicts with \eqref{eq:contradiction}, so we prove that $D$ is not of minimal type.

It remains to show that $\Lambda^*=\frac{1}{2}\Lambda_+$.
According to equation \eqref{l+}, we have 
\begin{equation*}
\langle \frac{1}{\sqrt{2}}\lambda,\frac{1}{\sqrt{2}} x\rangle\in\mathbb{Z}, \ \forall \ \lambda\in \Lambda_+,\ x\in X,
\end{equation*}
implying that $\frac{1}{\sqrt{2}}\Lambda$ is in the dual lattice of $\Gamma=\frac{1}{\sqrt{2}}\Lambda_+$. Thus, we have 
\begin{equation*}
\Gamma=\frac{1}{\sqrt{2}}\Lambda_+
\subsetneqq \frac{1}{\sqrt{2}}\Lambda\subset\Gamma^*.
\end{equation*}
On the other hand, according to Lemma \ref{lemma1}, we have $\Gamma^*\slash\Gamma\cong\mathbb{Z}\slash2\mathbb{Z}$.
This means that if a sublattice of $\Gamma^*$ contains $\Gamma$ and is not equal to $\Gamma$, then it must be $\Gamma^*$ itself.
Therefore, we have 
$\frac{1}{\sqrt{2}}\Lambda=\Gamma^*$. Then, it follows that $(\frac{1}{\sqrt{2}}\Lambda)^*=(\Gamma^*)^*=\Gamma=\frac{1}{\sqrt{2}}\Lambda_+$. Since $(\frac{1}{\sqrt{2}}\Lambda)^*=\sqrt{2}\Lambda^*$, we obtain $\sqrt{2}\Lambda^*=\frac{1}{\sqrt{2}}\Lambda_+$, meaning that $\Lambda^*=\frac{1}{2}\Lambda_+$. 
This completes the proof.

\end{proof}

We prove that infinitely many integers $m$ satisfy the assumptions in Theorem~\ref{th1}.  
\begin{theorem}
    $|\{m\in \mathbb{N} \mid m\equiv1\pmod{2},m\not\equiv 1 \pmod{3},m+1\not \equiv 0\pmod{8},m(m+1)\text{ is square-free}\}|=\infty$. 
\end{theorem}
\begin{proof}
We follow the idea of proof of \cite[Theorem~2]{HB}. 

Let $f(x)=|\{m\in \mathbb{N} \mid m\leq x,m\equiv1\pmod{2},m\not\equiv 1 \pmod{3},m+1\not \equiv 0\pmod{8},m(m+1)\text{ is square-free}\}|$. 
Define a function $E$ on the set of positive integers by $E(n)=1$ if $n$ is square-free and $E(n)=0$ otherwise. 
Set $Z(x)=\{m\in \mathbb{N} \mid m\leq x, m\equiv1\pmod{2},m\not\equiv 1 \pmod{3},m+1\not \equiv 0\pmod{8},m(m+1)\text{ is square-free}\}$ and 
 $Z'(x)=\{m\in \mathbb{N} \mid m\leq x, m\equiv1\pmod{8},m\equiv 0 \pmod{3},m(m+1)\text{ is square-free}\}$.  
Then 
\begin{equation}\label{eq:o1}
f(x)=\sum_{m\in Z(x)}E(m)E(m+1)
\geq \sum_{m\in Z'(x)}E(m)E(m+1)
=\sum_{i,j\leq x}\mu(i)\mu(j)N(x,i,j), 	
\end{equation}
where $N(x,i,j)=|\{k\in\mathbb{N} \mid k\leq x, k\equiv0 \pmod{3}, k\equiv 1\pmod{8}, i^2|k,j^2|(k+1)\}|$ and $\mu$ denotes the M\"{o}bius function. 
By the Chinese remainder theorem,  
$N(x,i,j)=\frac{x}{24i^2j^2}-$ if $(i,j)=1$ and $N(x,i,j)=0$ otherwise. 
Let $y=\sqrt{x}$, and consider the sum in \eqref{eq:o1} with the terms with $ij\leq y$; 
\begin{align*}
    &x\sum_{ij\leq y,(i,6)=(j,6)=1}\frac{\mu(i)\mu(j)}{24i^2j^2}+O(\sum_{ij\leq y})\\
    &=x\sum_{ij\leq y,(i,6)=(j,6)=1}\frac{\mu(i)\mu(j)}{24i^2j^2}+O(\sum_{ij\leq y})\\
    &=x\sum_{(i,6)=(j,6)=1}\frac{\mu(i)\mu(j)}{24i^2j^2}+O(x\sum_{n> y}\frac{d(n)}{n^2})+O(\sum_{n\leq y}d(n))\\
    &=\frac{C}{24}x+O(xy^{-1}\log y)+O(y\log y)\\
    &=\frac{C}{24}x+O(\sqrt{x}\log x),
\end{align*}
where $C=\prod_{p\in \mathbb{P}\setminus\{2,3\}}(1-\frac{2}{p^{-2}})=0.82963\cdots$ and $\mathbb{P}$ denotes the set of the prime numbers.
Then the result follows. 
\end{proof}

\section{Antipodal spherical $4$-distance $5$-designs of minimal type}\label{section4}

In this section we consider the existence of antipodal spherical $4$-distance $5$-design of minimal type. It is well known that $D\subset \mathbb{S}^{d-1}$ forms an antipodal spherical $4$-distance $5$-design if and only if $D=X\cup -X$ where $X$ is a Levenstein-equality packing in $\mathbb{R}^d$ \cite[Example 8.4]{DGS77}. 
Recall that if $n>\frac{d(d+1)}{2}$ then the coherence $\mu(X)$ of a finite set $X=\{x_i\}_{i=1}^{n}\subset \mathbb{S}^{d-1}$ satisfies the Levenstein bound \cite{leven,leven2}:
\begin{equation}\label{leven}
\mu(X)\geq\alpha_{n,d}:= \sqrt{\frac{3n-d(d+2)}{(d+2)(n-d)}}.
\end{equation}
A set $X=\{x_i\}_{i=1}^{n}\subset \mathbb{S}^{d-1}$ attaining the Levenstein bound \eqref{leven} is called a Levenstein-equality packing with parameters $(d,n)$. 
For a Levenstein-equality packing $X$ with parameters $(d,n)$, we have $A(X)=\{0,\pm\alpha_{n,d}\}$ and $n\leq \frac{d(d+1)(d+2)}{6}$. 

\subsection{Equivalent conditions for antipodal spherical $4$-distance $5$-designs to be of minimal type}

We first give an equivalent condition when an antipodal spherical $4$-distance $5$-design is of minimal type.

\begin{theorem}\label{thm:levenstein}
Let $d>4$. 
The existence of the following are equivalent.
\begin{enumerate}
\item[{\rm (i)}] An antipodal spherical $4$-distance $5$-design in $\mathbb{S}^{d-1}$
of size $2n$ and of minimal type.
\item[{\rm (ii)}] Spherical $3$-designs $X_1,X_2,X_3$ in $\mathbb{S}^{d-2}$ with $X_1=-X_3,X_1\cap -X_1=\emptyset, X_2=-X_2$, $|X_1|=|X_3|=\frac{(d+2)n}{3d}, |X_2|=\frac{4(d-1)n}{3d}$, and 
\begin{align*}
&A(X_1),A(X_3)\subset\left\{-\frac{3}{d-1},\frac{(d+2)\cdot \alpha_{n,d}-3}{d-1},\frac{-(d+2)\cdot \alpha_{n,d}-3}{d-1} \right\},\ A(X_2)\subset \left\{-1,0,\pm\alpha_{n,d}\right\},\\
&A(X_1,X_2),A(X_2,X_3)\subset\left\{0,\pm \sqrt{\frac{d+2}{d-1}}\cdot \alpha_{n,d}\right\},\\
& A(X_1,X_3)\subset\left\{-1,\frac{3}{d-1},\frac{(d+2)\cdot \alpha_{n,d}+3}{d-1},\frac{-(d+2)\cdot \alpha_{n,d}+3}{d-1}\right\},
\end{align*}
where $\alpha_{n,d}$ is defined in \eqref{leven}.

\end{enumerate}

\end{theorem}

\begin{proof}
(i)$\Rightarrow$(ii):
Let $D=X\cup -X$ be an antipodal spherical $4$-distance $5$-design in $\mathbb{S}^{d-1}$ of minimal type, where $X$ is a Levenstein-equality packing with size $n$. 
Since $D$ is of minimal type, there exists  $ \alpha\in \mathbb{R}^d$ such that $\langle  \alpha, x\rangle\in \{0,\pm1\}$ for all $x\in D$.
By a similar calculation used in Theorem \ref{thm:cc} (i), we have $\langle \alpha,\alpha\rangle=\frac{d+2}{3}$ and
\begin{equation*}
n_0( \alpha)=\frac{4(d-1)n}{3d}\quad\text{and}\quad n_1( \alpha)=\frac{2(d+2)n}{3d}, 
\end{equation*}
where $n_\ell(\alpha):=|\{x\in D \mid \langle \alpha,x\rangle=\pm \ell\}|$, $l\in\{0,1\}$.
After suitably transforming the set $D$ and the vector $\alpha$, we may assume that $\alpha=(\sqrt{\frac{d+2}{3}},0,\ldots,0)$.
Let $X_1=\mathcal{L}_{\alpha,1}(D)$, $X_2=\mathcal{L}_{\alpha,0}(D)$ and $X_3=\mathcal{L}_{\alpha,-1}(D)$.
Since $D$ is antipodal, it follows that $|X_1|=|X_3|=n_1( \alpha)/2$ and $|X_2|= n_0( \alpha)$. By Theorem \ref{thm:derived} we see that each $X_i$ is a spherical $3$-design in $\mathbb{S}^{d-2}$, and their angle sets satisfy the desired conditions in (ii).

(ii)$\Rightarrow$(i): We proceed the converse implication above. 
Set $D=\bigcup_{i=1}^3 \widetilde{X}_i$, where 
\begin{equation*}
\widetilde{X}_1=\bigg\{\Big(\sqrt{\tfrac{3}{d+2}},\sqrt{\tfrac{d-1}{d+2}}\cdot x\Big) \ \Big|\  x\in X_1\bigg\},\  \widetilde{X}_2=\bigg\{(0,x) \ \Big|\  x\in X_2\bigg\},\ \widetilde{X}_3=-\widetilde{X}_1.
\end{equation*}
Then $D$ is an antipodal spherical $4$-distance set in $\mathbb{S}^{d-1}$ with $A(D)=\{-1,0,\pm\alpha_{n,d}\}$ and $|D|=2n$. Since a half of $D$ attains the Levenstein bound \eqref{leven}, $D$ forms an antipodal spherical $4$-distance $5$-design in $\mathbb{S}^{d-1}$. 
This completes the proof.
\end{proof}

\begin{remark}
Set 
\begin{align*}
    \left(\alpha_{k}^{(1,1)}\right)_{k=1}^3&=\left(-\frac{3}{d-1},\frac{(d+2)\cdot \alpha_{n,d}-3}{d-1},\frac{-(d+2)\cdot \alpha_{n,d}-3}{d-1}\right),\\
    \left(\alpha_{k}^{(2,2)}\right)_{k=1}^3&=\left(0,\alpha_{n,d},-\alpha_{n,d}\right),\\
    \left(\alpha_{k}^{(i,j)}\right)_{k=1}^3&=\left(0, \sqrt{\frac{d+2}{d-1}}\cdot \alpha_{n,d},- \sqrt{\frac{d+2}{d-1}}\cdot \alpha_{n,d}\right)\quad \text{ for }(i,j)\in\{(1,2),(2,1)\}. 
\end{align*}
Write a valency $p_{\alpha}^j$ from a point in $X_i$ as $p_{i,j,k}$ where $\alpha=\alpha_k^{(i,j)}$.   
By applying Theorem~\ref{thm:di} to $X_1$ and $X_2$ in Theorem~\ref{thm:levenstein}, we obtain that $X_1$ and $X_2$ are distance invariant and the following valencies: 
\begin{align*}
    p_{1,1,1}&=\frac{(d-1) n \left(d^2+d-2 n\right)}{3 d (d (d+2)-3 n)},\\
    p_{1,1,2}&=\frac{(n-d) (d (-3 d+n-6)+8 n)-3 (d+2) (d-n)\alpha_{n,d})}{6 d (3 n-d (d+2))}, \\
    p_{1,1,3}&=\frac{(n-d) (d (-3 d+n-6)+8 n)+3 (d+2)  (d-n)\alpha_{n,d})}{6 d (3 n-d (d+2))},\\
    p_{1,2,1}&=\frac{4 (d-1) n \left(d^2+d-2 n\right)}{3 d (d (d+2)-3 n)},\\
    p_{1,2,2}&=p_{1,2,3}=\frac{2 (d-1) n (d-n)}{3 d (d (d+2)-3 n)},\\
    p_{2,1,1}&=\frac{(d+2) n \left(d^2+d-2 n\right)}{3 d (d (d+2)-3 n)},\\
    p_{2,1,2}&=p_{2,1,3}=\frac{(d+2) n (d-n)}{6 d (d (d+2)-3 n)},\\
    p_{2,2,1}&=\frac{2 (2 d-5) n \left(d^2+d-2 n\right)}{3 d (d (d+2)-3 n)},\\
    p_{2,2,2}&=p_{2,2,3}=-\frac{(d+2) (3 d-2 n) (d-n)}{3 d (d (d+2)-3 n)}.
\end{align*}
Hence, if one of the above valencies is not integal, then there does not exist an antipodal spherical $4$-distance $5$-design in $\mathbb{S}^{d-1}$ of minimal type that has size $2n$. 
\end{remark}

\subsection{List of all known antipodal spherical $4$-distance $5$-designs of minimal type}

Throughout this section, we let $D\subset\mathbb{R}^d$ be an antipodal spherical $4$-distance $5$-design.
Then $D$ has the form $D=X\cup -X$, where $X\subset\mathbb{S}^{d-1}$ is a Levenstein-equality. 
Table \ref{Levenstein-equality packing} lists all known constructions of antipodal spherical $4$-distance $5$-designs (see also \cite{Haas,Mun06}).
It is worth noting that only finitely many constructions are known when $|D|\neq d(d+2)$.
In what follows, we examine whether these known examples are of minimal type.

We first consider the case when $D$ forms a tight spherical $7$-design. 
In this case, we have $|D|=\frac{d(d+1)(d+2)}{3}$.
It is well known that tight spherical $7$-designs may exist only when $d=3k^2-4$, $k\in\mathbb{Z}$, and the existence is known only for $d=8$ and $d=23$ \cite{BMV04,BB09b}. 
In the following theorem we show that $D$ is not of minimal type.

\begin{theorem}\label{tight7design}
Let $d>1$.
Then tight spherical $7$-designs in $\mathbb{R}^d$ are not of minimal type.
	
\end{theorem}

\begin{proof}
We prove by contradiction.
Let  $D\subset\mathbb{S}^{d-1}$ be a tight spherical $7$-design with size $|D|=\frac{d(d+1)(d+2)}{3}$. Then for any $y\in\mathbb{R}^d$ we have \cite{BMV04}
\begin{subequations}
\begin{align}
\sum_{x\in D}\langle y,x \rangle^2&=\frac{1}{d}\cdot |D|\cdot \langle y,y \rangle,\label{sperical5design-22} \\
\sum_{x\in D}\langle y,x \rangle^4&=\frac{3}{d(d+2)}\cdot |D|\cdot\langle y,y \rangle^2,\label{sperical5design-42}\\
\sum_{x\in D}\langle y,x \rangle^6&=\frac{15}{d(d+2)(d+4)}\cdot |D|\cdot\langle y,y \rangle^3.\label{sperical5design-62}
\end{align}
\end{subequations}
Assume that $D$ is of minimal type. 
Then there exists $\alpha\in\mathbb{R}^d$  
such that $\langle  \alpha, x\rangle\in \{0,\pm1\}$ for all $x\in D$. Define $n_\ell(\alpha):=|\{x\in D \mid \langle \alpha,x\rangle=\pm \ell\}|$ for $l\in\{0,1\}$.
By equation \eqref{sperical5design-22} and \eqref{sperical5design-42} we have $\langle \alpha,\alpha\rangle=\frac{d+2}{3}$ and $n_1( \alpha)=\frac{d+2}{3d}\cdot |D|$.
Substituting $y=\alpha$ and $\langle \alpha,\alpha\rangle=\frac{d+2}{3}$ into \eqref{sperical5design-62} we obtain
\begin{equation*}
n_1(\alpha)=\frac{15}{d(d+2)(d+4)}\cdot |D|\cdot(\frac{d+2}{3})^3	=\frac{5(d+2)^2}{9d(d+4)}\cdot |D|,
\end{equation*}
which contradicts with $n_1( \alpha)=\frac{d+2}{3d}\cdot |D|$ when $d>1$. 
Therefore, $D$ is not of minimal type.
	
\end{proof}

We next dealt with the remaining case where $4\leq d\leq 23$. 

\begin{example}\label{antipodal-example}
Assume $(d,|D|)\in \{(4,24), (6,72),(7,126),(8,240),(16,288),(22,2816),(23,4600)\}$. 
In each of these cases, the set $D$ corresponds to  the set of the shortest vectors of a strongly perfect lattice, as listed in Table \ref{Levenstein-equality packing}.
These strongly perfect lattices are of minimal type,  since their Berg\'e-Martinet invariants achieve the lower bound \eqref{minimal type lattice} \cite{Mar13,HN20,CS88}.
Hence, by Remark \ref{remark-strongly perfect lattices},  the associated antipodal spherical $4$-distance $5$-designs are also of minimal type.
\end{example}

\begin{example}
\label{example-16-144}
Assume $(d,|D|)=(16,288)$.
In this case, a half of $D$ forms the maximal mutually unbiased bases, which consists of the vectors of the standard basis and vectors obtained from the Nordstrom-Robinson code by changing $0$ by $-1$ \cite{CCKS97}. 
Therefore, one half of $D$ can be represented as the columns of the matrix $X=[B_0, B_1,\ldots, B_8]\in\mathbb{R}^{16\times 144}$.
Here, $B_0=I$ is the identity matrix, and $B_i=[b_{i,1},\ldots,b_{i,16}]\in \{\pm \frac{1}{4}\}^{16\times 16}$ satisfies $B_i^\top B_i=I$ for each $1\leq i\leq 8$.
We now prove that the set $D$ is not of minimal type by contradiction. 
Suppose, for contradiction, that $D$ is of minimal type. 
Then there exists a vector $\alpha\in\mathbb{R}^{16}$ such that $\langle \alpha, x\rangle\in\{0,\pm 1\}$ for any $x\in D$.
Since $D$ is constructed from the Nordstrom-Robinson code, $\alpha$ must lie in the set
\begin{equation*}
S:=\{ y=(y_1,\ldots,y_{16})\in\mathbb{R}^{16}\mid  y_i\in\{0,\pm 1\},\ \forall 1\leq i\leq 16\  \text{and}\  \|y\|_2^2=6 \}.
\end{equation*}
However, through exhaustive enumeration, one can verify that no vector $\widetilde{\alpha}\in S$ satisfies $\langle \widetilde{\alpha}, b_{i,j} \rangle\in\{0,\pm 1\}$ for each $1\leq i\leq 8$ and for each $1\leq j\leq 16$.
This contradiction implies that such a vector $\alpha$ does not exist, and therefore, $D$ is not of minimal type.

\end{example}

When $n=d(d+2)$ and $d=4^s$ for some integer $s\geq 1$, antipodal spherical $4$-distance $5$-designs are equivalent to maximal real mutually unbiased bases (MUBs). 
Such configurations are known to exist for each integer $s\geq 1$ \cite{CCKS97}. 
Example \ref{antipodal-example} and Example \ref{example-16-144} show that an antipodal spherical $4$-distance $5$-design $D$ is of minimal type when $d=4$, and not of minimal type when $d=16$.
It remains an open question whether $D$ is of minimal when $d=4^s>16$.

\bigskip
\noindent
\textbf{Acknowledgments.}
Sho Suda is supported by JSPS KAKENHI Grant Number 22K03410. 
Wei-Hsuan Yu is supported by  MOST  in  Taiwan  under  Grant  MOST109-2628-M-008-002-MY4.

\Addresses

\begin{thebibliography}{99}

\bibitem[BB09a]{BB09a}
Ei. Bannai and Et. Bannai, 
\newblock {\em  On antipodal spherical $t$-designs of degree $s$ with $t\geq 2s-3$,}
\newblock {\em  J. Comb. Inf. Syst. Sci.}, 34: 33-50, (2009).

\bibitem[BB09b]{BB09b}
Ei. Bannai and Et. Bannai, 
\newblock {\em  A survey on spherical designs and algebraic combinatorics on spheres,}
\newblock {\em  European J. Combin.}, 30(6), pp.1392-1425.

\bibitem[BI84]{BI84}
E. Bannai and T. Ito,
\newblock {\em  Algebraic Combinatorics I: Association Schemes,}
\newblock {\em  Benjamin/Cummings, Menro Park, CA,}  1984.


\bibitem[BKN26]{BKN26}
E. Bannai, H. Kurihara, and H. Nozaki, 
\newblock {\em  On the existence and non-existence of spherical $m$-stiff configurations,}
\newblock {\em Discrete Math.}, 349(1):114731, (2026).


\bibitem[BMV04]{BMV04}
E. Bannai, A. Munemasa, and B. Venkov,
\newblock {\em  The nonexistence of certain tight spherical designs (with an appendix by Y.-F. S. Pétermann),}
\newblock {\em  St. Petersburg Math. J.},  16.4:609-625, (2005). 


\bibitem[BGOY15]{BGOY15}
A. Barg, A. Glazyrin,  K.A. Okoudjou, and W.H. Yu,
\newblock {\em  Finite two-distance tight frames,}
\newblock {\em  Linear Algebra Appl.}, 475:163-175, (2015).


\bibitem[B22]{B22}
S. Borodachov,  
\newblock {\em  Absolute minima of potentials of a certain class of spherical designs,}
\newblock {\em  }  arXiv:2212.04594.

\bibitem[B24]{B24}    
S. Borodachov, 
\newblock {\em Odd strength spherical designs attaining the Fazekas–Levenshtein bound for covering and universal minima of potentials,}
\newblock {\em Aequationes Mathematicae},   98(2), 509--533, (2024). 

\bibitem[CCKS97]{CCKS97}
A. R. Calderbank, P. J. Cameron, W. M. Kantor, and J. J. Seidel, 
\newblock {\em $\mathbb{Z}_4$-Kerdock codes, orthogonal spreads, and extremal Euclidean line-sets,}
\newblock {\em Proc. London Math. Soc.}, (3)75:436-480, (1997).

\bibitem[CHS96]{Conway}
J. H. Conway, R. H. Hardin, and N. J. A. Sloane,
\newblock {\em Packing lines, planes, etc.: packings in Grassmannian spaces,}
\newblock  {\em Experiment. Math.}, 5(2):139-159 (1996).


\bibitem[CS88]{CS88}
J. H. Conway, and N. J. A. Sloane,
\newblock {\em Low-dimensional lattices. II. Subgroups of $GL (n, \mathbb{Z})$,}
\newblock  {\em Proceedings of the Royal Society of London. A. Mathematical and Physical Sciences}, 419(1856), 29-68, (1988).


\bibitem[Del73]{D73}
P. Delsarte,
\newblock {\em An algebraic approach to the association schemes of coding theory,} 
\newblock {\em Philips Res. Rep. Suppl.}, 10:vi+-97, (1973).

\bibitem[DGS77]{DGS77}
P. Delsarte, J. M. Goethals, and J. J. Seidel,
\newblock {\em  Spherical codes and designs,}
\newblock {\em  Geom. Dedicata},   6(3):363-388, (1977).

\bibitem[FL95]{FL95}
G. Fazekas and V. I. Levenshtein, 
\newblock {\em  On upper bounds for code distance and covering radius of designs in polynomial metric spaces,}
\newblock {\em  J. Combin. Theory Ser. A},  70(2), 267-288, (1995).


\bibitem[FJM18]{Fickus2}
M. Fickus, J. Jasper, and D. G. Mixon,
\newblock {\em Packings in real projective spaces,}
\newblock {\em SIAM J. Appl. Algebra Geom.}, 2(3):377-409, (2018).


\bibitem[GY18]{GY18}
A. Glazyrin and W.H. Yu,
\newblock {\em  Upper bounds for $s$-distance sets and equiangular lines,}
\newblock {\em  Adv. Math.}, 330:810-833, (2018).


\bibitem[HHM17]{Haas}
J. I. Haas, N. Hammen, and D. G. Mixon,
\newblock {\em The Levenstein bound for packings in projective spaces, }
\newblock {\em Wavelets and Sparsity XVII,} Vol. 10394, p. 103940V (24 August 2017), International Society for Optics and Photonics. https://doi.org/10.1117/12.2275373

\bibitem[Hea84]{HB}
D. R. Heath-Brown, 
\newblock {\em The square sieve and consecutive square-free numbers,}
\newblock {\em Math. Ann.}, 266:251-259, (1984).


\bibitem[Hig75]{Hig75}
D. G. Higman,
\newblock {\em  Coherent configurations I. Ordinary representation theory,}
\newblock {\em  Geom. Dedicata}, 4:1-32, (1975).

\bibitem[HN20]{HN20}
S. Hu and G. Nebe,
\newblock {\em  Low dimensional strongly perfect lattices IV: The dual strongly perfect lattices of dimension 16,}
\newblock {\em  J. Number Theory}, 208:262-294, (2020).



\bibitem[JKM19]{Jasper}
J. Jasper, E. J.  King, and D. G. Mixon,
\newblock {\em  Game of Sloanes: best known packings in complex projective space,}
\newblock {\em Wavelets and Sparsity XVIII,}  vol. 11138, p. 111381E. International Society for Optics and Photonics, (2019).


\bibitem[Lev92]{leven}
V. I. Levenshtein,
\newblock {\em Designs as maximum codes in polynomial metric spaces,}
\newblock {\em Acta Appl. Math.}, 29:1-82, (1992).

\bibitem[Lev98]{leven2}
V. I. Levenshtein,
\newblock {\em Universal bounds for codes and designs,}
\newblock {\em Handbook of coding theory,} 1:499-648, (1998).

\bibitem[LY20]{LY}
Y. C. R. Lin and  W. H. Yu,
\newblock {\em Equiangular lines and the Lemmens-Seidel conjecture,} 
\newblock {\em Discrete Math.}, 343(2):111667, (2020).

\bibitem[Mak02]{Mak02}
A. A. Makhnev, 
\newblock {\em  On the nonexistence of strongly regular graphs with parameters (486, 165, 36, 66),}
\newblock {\em  Ukrainian Math. J.}, 54(7):1137-1146, (2002).

\bibitem[Mar13]{Mar13}
J. Martinet, 
\newblock {\em  Perfect lattices in Euclidean spaces,}
\newblock {\em  Springer Science \& Business Media.}  Vol. 327.


\bibitem[Mun06]{Mun06}
A. Munemasa, 
\newblock {\em  Spherical Designs, in: Handbook of Combinatorial Designs, 2nd ed.,}
\newblock {\em  CRC Press,} pp. 617-622, (2006).

\bibitem[NV00]{NV00}
G. Nebe and B. Venkov,
\newblock {\em  The strongly perfect lattices of dimension $10$,}
\newblock {\em  Journal de th\'eorie des nombres de Bordeaux,} 12(2):503-518, (2020).


\bibitem[NV13]{NV13}
G. Nebe and B. Venkov,
\newblock {\em  On tight spherical designs,}
\newblock {\em  St. Petersburg Math. J.}, 24.3:485-491, (2013).

\bibitem[Neu84]{Neumaier}
A. Neumaier,
\newblock {\em Some sporadic geometries related to PG (3,2),} 
\newblock {\em Arch. Math. (Basel)}, 42:89-96, (1984).

\bibitem[Suda10]{Suda10}
S. Suda,
\newblock {\em  Coherent configurations and triply regular association schemes obtained from spherical designs,}
\newblock {\em  J. Combin. Theory Ser. A,} 8:1178-1194,  (2010). 

\bibitem[Suda22]{Suda22}
S. Suda,
\newblock {\em  $Q$-polynomial coherent configurations,}
\newblock {\em  Linear Algebra Appl.}, 643:166-195, (2022).


\bibitem[Ven01]{Ven01}
B. Venkov,
\newblock {\em  R\'eseaux et designs sph\'eriques,}
\newblock {\em  R\'eseaux euclidiens, designs sph\'eriques et formes modulaires,} 37:10-86,  (2001).

\bibitem[Wel74]{Welch}
L. Welch,
\newblock {\em Lower bounds on the maximum cross correlation of signals,}
\newblock {\em IEEE Trans. Inform. Theory.,} 20(3):397-399, (1974).


\bibitem[XXY21]{XXY21}
Z. Xu, Z. Xu, and  W. H. Yu,
\newblock {\em  Bounds on antipodal spherical designs with few angles,}
\newblock {\em Electron. J. Combin.}, P3-39, (2021).

\end{thebibliography}
\end{document}